\documentclass{siamltex}
\usepackage{amsmath}
\usepackage{amssymb}
\usepackage{epsfig,overpic}
\newcommand{\Div}{\operatorname{div}}
\newcommand{\bn} {\mathbf{n}}
\newcommand{\bx} {\mathbf{x}}
\newcommand{\by} {\mathbf{y}}

\newcommand{\bp} {\mathbf{p}}
\newcommand{\bs} {\mathbf{s}}
\newcommand{\br} {\mathbf{r}}
\newcommand{\bw} {\mathbf{w}}
\newcommand{\bH} {\mathbf{H}}
\newcommand{\bA} {\mathbf{A}}
\newcommand{\bB} {\mathbf{B}}
\newcommand{\bP} {\mathbf{P}}
\newcommand{\bV} {\mathbf{V}}
\newcommand{\bVt} {\widetilde{\mathbf{V}}}

\newcommand{\deltah} {\delta_h}
\newcommand{\eps} {\varepsilon}

\newcommand{\dO} {{\partial\Omega}}
\def\unabla{\nabla_{\Gamma}}
\def\unablah{\nabla_{\Gamma_h}}
\newcommand{\T}{\mathcal T}

\newcommand{\bI}{\mathbf I}
\newtheorem{remark}{Remark}

\title{A trace finite element method for a class of coupled bulk-interface transport problems
\thanks{Partially supported by NSF through the Division of Mathematical Sciences grant 1315993.}}
\author{Sven Gross\thanks{Institut f\"ur Geometrie und Praktische  Mathematik, RWTH-Aachen
University, D-52056 Aachen, Germany (gross@igpm.rwth-aachen.de)} \and Maxim A. Olshanskii\thanks{Department of Mathematics, University of Houston, Houston, Texas 77204-3008
(molshan@math.uh.edu).}
\and Arnold Reusken\thanks{Institut f\"ur Geometrie und Praktische  Mathematik, RWTH-Aachen
University, D-52056 Aachen, Germany (reusken@igpm.rwth-aachen.de).}
}

\begin{document}

\maketitle

\begin{abstract}
In this  paper we  study a system of advection-diffusion equations  in a bulk domain coupled to an advection-diffusion equation on an embedded surface. Such systems of coupled partial differential equations arise in, for example, the modeling of transport and diffusion of surfactants in two-phase flows. The model considered here accounts for adsorption-desorption of the surfactants  at a sharp interface between two fluids and their transport and diffusion in both fluid phases and along the interface. The paper gives a well-posedness analysis for the system of bulk-surface equations and introduces a finite element method
for its numerical solution. The finite element method is unfitted, i.e., the mesh is not aligned to the interface. The method is based on taking traces of a standard finite element space both on the bulk domains and the embedded surface. The numerical approach allows an implicit definition of the surface as the zero level of a level-set function. Optimal order error estimates are proved
for the finite element method both in  the bulk-surface energy norm  and the $L^2$-norm. The analysis is not restricted to linear finite elements and a piecewise planar reconstruction of the surface, but also  covers the discretization with higher order elements and a higher order surface reconstruction.
\end{abstract}

\section{Introduction}

Coupled bulk-surface or bulk-interface partial \  differential equations arise in many applications, e.g.,  in multiphase fluid dynamics \cite{GReusken2011} and biological applications \cite{Bonito}.
In this paper, we consider a coupled bulk-interface advection-diffusion problem. The problem arises in models describing the behavior of soluble surface active agents (surfactants) that are adsorbed at liquid-liquid interfaces.
For a discussion of physical phenomena related to soluble surfactants in two-phase incompressible flows we refer to the literature, e.g., \cite{GReusken2011,Ravera,Clift,Tasoglu}.

Systems of partial differential equations that couple bulk domain effects with interface (or surface) effects pose  challenges both for the  mathematical analysis of equations and the development and error analysis of numerical methods. These  challenges grow if  phenomena  occur at different physical scales, the coupling is nonlinear or the interface is evolving in time. To our knowledge, the analysis of numerical methods for coupled bulk-surface (convection-)diffusion has been addressed in the literature only very recently. In fact, problems related to the one studied in this paper have been considered only in \cite{BurmanHansbo,ElliottRanner}. 
In these references finite element methods for coupled bulk-surface partial differential equations are proposed and analyzed. In \cite{BurmanHansbo,ElliottRanner} a stationary diffusion problem on a bulk domain is linearly coupled with a stationary diffusion equation on the boundary of this domain. A key difference between the methods in \cite{BurmanHansbo} and \cite{ElliottRanner} is that in the latter boundary \emph{fitted} finite elements are used, whereas in the former \emph{unfitted} finite elements  are applied. Both papers include error analyses of these methods. In the recent paper \cite{Chen2014} a similar coupled surface-bulk system is treated with a  different approach, based on the immersed boundary method. In that paper an evolving surface is considered, but only spatially two-dimensional problems are treated and no theoretical error analysis is given.

In the present paper, as in \cite{BurmanHansbo,ElliottRanner} we restrict to stationary problems and a linear coupling. The results obtained are a starting point for research on other classes of problems, e.g., with an evolving interface, cf. the discussion in section~\ref{sectdisc}.
The two main new contributions of this paper are the following. Firstly, the  class of problems considered is significantly different from the one treated in \cite{BurmanHansbo,ElliottRanner}. We study a problem in which elliptic partial differential equations in  \emph{two} bulk subdomains are coupled to a partial differential equation on a sharp \emph{interface} which separates the two subdomains. In the previous work only a coupling of an elliptic partial differential  in \emph{one} bulk domain with a partial differential equation on the \emph{boundary} of this bulk domain is treated. Furthermore, the  partial differential equations considered in this paper are not pure diffusion equations, but \emph{convection}-diffusion equations. The latter model the transport by advection and diffusion of surfactants. We will briefly address some basic modeling aspects, e.g. related to adsorption (Henry and Langmuir laws), of these coupled equations. The first main new result is the
well-posedness of a weak
formulation of this coupled system.
We introduce suitable function spaces and an appropriate weak formulation of the problem. We derive a Poincare type inequality in a bulk-interface product Sobolev space and show
an inf-sup stability result for the bilinear form of the weak formulation. This then leads to the well-posedness result.
The second main new contribution is the error analysis of a finite element method. We consider an \emph{unfitted} finite element method, which is very similar  to the method presented in \cite{BurmanHansbo}. In the method treated in this paper we do not apply the stabilization technique used in  \cite{BurmanHansbo}.  An unfitted approach is particularly attractive for problems with an evolving interface, which will be studied in a follow-up paper.
 Both interface and bulk
finite element spaces are trace spaces of globally defined continuous finite element functions with respect to a
regular simplicial triangulation of the whole domain. For bulk problems, such finite element techniques have been extensively studied in the
literature on  cut finite element methods or XFEM, cf. e.g. \cite{Hansbo02,Hansbo04,Hansbo2012}. For
PDEs posed on surfaces, the trace finite element method was introduced and studied in \cite{OlshReusken08,ORX,OlshanskiiReusken08,DemlowOlshanskii12}.
In the method that we propose, the smooth interface is approximated by a piecewise smooth one, characterized by the zero level set of a finite element level set function. This introduces a geometric error in the method. The approach
allows  meshes that do not fit to this (approximate) interface and admits implicitly defined interfaces.
The finite element formulation is shown to be well-posed.  We present an error analysis that is general in the sense that finite element polynomials of arbitrary degree are allowed  (in  \cite{BurmanHansbo} only linear finite elements are treated) and that the accuracy of the interface approximation can be varied. The error analysis is rather technical and we aimed at a clear exposition by subdividing the analysis into several steps: the construction of a bijective mapping between
the continuous bulk domains and their numerical approximations, the definition and analysis of extensions of functions off the interface and outside the bulk domains, and the analysis of  consistency terms in an approximate Galerkin orthogonality property. This leads to an optimal order bound for the discretization error in the energy norm of the product space. Finally, by using a suitable adjoint problem, we derive an optimal  order  error
bound in the $L^2$ product norm. Results of numerical experiments are included that illustrate the convergence behavior of the finite element method. A point that we do not address in this paper is the stabilization of the discretization method with respect to the conditioning of the stiffness matrix. From the  literature it  is known that the trace finite element approach results in a poor conditioning of the stiffness matrix. The technique presented in \cite{Hansbo2012,BurmanHansbo} can be used (for linear finite elements) to obtain a stiffness matrix with conditioning properties similar to those of a standard finite element method. We think that this technique is applicable also to the method presented in this paper, but decided not to include it, since the additional stabilization terms would lead to a further increase of technicalities in the analysis. Finally we mention the recent related paper \cite{Deckelnick2013}, in which unfitted finite element techniques,  similar to the
 one used in this paper,
are applied to   surface partial differential equations (without coupling to bulk domains).

\section{Mathematical model}\label{S_model}

In this section we explain the physical background of the coupled bulk-interface model that we treat in this paper.
Consider a two-phase incompressible flow system in which two fluids occupy   subdomains $\Omega_i(t)$, $i=1,2$, of a given domain $\Omega \subset \mathbb{R}^3$.
The corresponding velocity field of the fluids is denoted by $\bw(\bx,t)$, $\bx \in \Omega$, $t\in [0,t_1]$.
For convenience, we assume that $\Omega_1(t)$ is simply connected and strictly contained in $\Omega$ (e.g., a rising droplet) for all $t\in [0,t_1]$. The outward pointing normal from $\Omega_1$ into $\Omega_2$ is denoted by $\bn$.  The sharp interface between the two fluids is denoted by $\Gamma(t)$. Below we write $\Gamma$ instead of $\Gamma(t)$. It is assumed that the fluids are immiscible and that there are no phase changes.  Consequently, the normal components of the fluid velocity are continuous at the interface and  the interface itself is advected with the flow, i.e.,  $V_\Gamma= \bw \cdot \bn$  holds, where $V_\Gamma$ denotes the normal velocity of the interface.   The standard  model for the fluid dynamics in such a system consists of the Navier-Stokes equations, combined with  suitable coupling conditions at the interface.  In the rest of this paper, we assume that the velocity field has smoothness $\bw(\cdot,t) \in  [H^{1,\infty}(\Omega)  \cap H^1(\Gamma)]^3$ and that
$\bw$ \emph{is given}, i.e.,
we do not consider a \emph{two}-way coupling between surfactant transport and fluid dynamics. The fluid is assumed incompressible:
\begin{equation}\label{div_w} \Div\bw=0   \quad \text{in}~~\Omega.
 \end{equation}

Consider a surfactant that is soluble in both phases and can be adsorbed and desorbed at the interface.
  The surfactant \emph{volume} concentration (i.e., the one in the bulk phases) is denoted by $u$, $u_i=u|_{\Omega_i}$, $i=1,2$. The surfactant \emph{area} concentration on $\Gamma$ is denoted by $v$. Change of the surfactant concentration happens due to convection by the velocity field $\bw$, diffusive fluxes in $\Omega_i$, a diffusive flux on $\Gamma$ and fluxes coming from adsorption and desorption. The net flux (per surface area) due to adsorption/desorption between $\Omega_i$ and $\Gamma$ is denoted by $j_{i,a}-j_{i,d}$. The total net flux is $j_a-j_d= \sum_{i=1}^2 (j_{i,a}-j_{i,d})$.
Mass conservation in a control volume (transported by the flow field) that is strictly contained in $\Omega_i$ results in the bulk convection-diffusion equation
\begin{equation} \label{diffeq1}
 \frac{\partial u}{\partial t} + \bw \cdot \nabla u- D_i \Delta u = 0 \quad \text{in}~ \Omega_i= \Omega_i(t), ~i=1,2.
\end{equation}
Here $D_i >0$ denotes the bulk diffusion coefficient, which is assumed to be constant in $\Omega_i$.
Mass conservation in a control area (transported by the flow field) that is completely contained in $\Gamma$ results in the surface convection-diffusion equation (cf.~\cite{GReusken2011}):
\begin{equation} \label{diffeq2}
 \dot v + (\Div_\Gamma \bw)v - D_\Gamma \Delta_\Gamma v = j_a-j_d  \quad \text{on}~~\Gamma= \Gamma(t),
\end{equation}
where $\dot v = \frac{\partial v}{\partial t} + \bw \cdot \nabla v$ denotes the material derivative and $\Delta_\Gamma$,  $\Div_\Gamma$  the Laplace-Beltrami and surface divergence operators, respectively. The interface diffusion coefficient $D_\Gamma >0$ is assumed to be a constant.

We assume that transport of surfactant between the two phases can only occur via adsorption/desorption.  Due to $V_\Gamma = \bw \cdot \bn$, the mass flux through $\Gamma$ equals the diffusive mass flux. Hence, mass conservation in a  control volume (transported by the flow field) that lies in $\Omega_i$ and with part of its boundary on  $\Gamma$ results in the mass balance equations
\begin{equation}\label{eq3}
   (-1)^{i} D_i \bn \cdot \nabla u_i = j_{i,a} -j_{i,d}, \quad i=1,2.
\end{equation}
The sign factor $(-1)^{i}$ accounts for the fact that the normal $\bn$ is outward pointing from $\Omega_1$ into $\Omega_2$. Summing these relations over $i=1,2$, yields
\[
   j_{a} -j_{d}=- [D \bn \cdot \nabla u]_\Gamma,
\]
where $[w]_\Gamma = (w_1)_{|\Gamma}-(w_2)_{|\Gamma}$ denotes the jump of $w$ across $\Gamma$. 

  To close the system of equations, we need constitutive equations for modeling the adsorption/desorption. A standard model, cf.~\cite{Ravera}, is as follows:
\begin{equation} \label{eq4}
 j_{i,a}-j_{i,d}=  k_{i,a} g_i(v) u_i - k_{i,d} f_i(v),\quad\text{on}~\Gamma,  
\end{equation}
with $k_{i,a}$, $k_{i,d}$ positive adsorption and desorption coefficients that describe the kinetics. We consider two fluids having similar adsorption/desorption behavior in the sense that the coefficients  $k_{i,a}$, $k_{i,d}$ may depend on $i$, but $g_i(v)=g(v)$, $f_i(v)=f(v)$ for $i=1,2$.  Basic choices for $g$, $f$ are the following:
\begin{align}
  & g(v)=1, \quad f(v)=v \quad \text{(Henry)} \label{Henry}\\
  & g(v)=1- \frac{v}{v_{\infty}},\quad f(v)=v \quad \text{(Langmuir)} \label{Langmuir},
\end{align}
where $v_\infty$ is a constant that quantifies the maximal concentration on $\Gamma$. Further options are given in \cite{Ravera}. Note that in the Langmuir case we have a nonlinearity due to the term $ v u_i $. 
Combining \eqref{diffeq1}, \eqref{diffeq2}, \eqref{eq3} and \eqref{eq4} gives a closed model. For the mathematical analysis it is convenient to  reformulate these equations in dimensionless variables. Let $L$, $W$ be appropriately defined length and velocity scales and $T=L/W$ the corresponding time scale. Furthermore, $U$ and $V$ are typical reference volume and area concentrations. The equations above can be reformulated in the dimensionless variables $\tilde x = x/L$, $\tilde t= t/T$, $\tilde u_i=u_i/U$, $\tilde v =v /V$, $\tilde \bw= \bw/W$. This results in the following system of \emph{coupled bulk-interface convection-diffusion equations}, where we use the notation $x,t,u_i,v,\bw$ also for the transformed variables:
\begin{equation} \label{total}
 \begin{split}
   \frac{\partial u}{\partial t} + \bw \cdot \nabla u- \nu_i \Delta u  & = 0 \quad \text{in}~ \Omega_i(t), ~i=1,2, \\
\dot v + (\Div_\Gamma \bw) v - \nu_\Gamma \Delta_\Gamma v  & =- K [ \nu \bn \cdot \nabla u]_\Gamma   \quad \text{on}~~\Gamma(t), \\
  (-1)^{i} \nu_i \bn \cdot \nabla u_i & = \tilde k_{i,a} \tilde g(v) u_i- \tilde k_{i,d} v \quad \text{on}~~\Gamma(t), \quad i=1,2, \\
\text{with}~~ \nu_i=\frac{D_i}{LW}, ~~\nu_\Gamma=\frac{D_\Gamma}{LW}, ~~K&=\frac{LU}{V},~~ \tilde k_{i,a} = \frac{T}{L}k_{i,a},~~\tilde k_{i,d}=\frac{T}{K}k_{i,d},
\end{split}
\end{equation}
and $\tilde g(v)=1$ (Henry) or $\tilde g(v)=1-\frac{V}{v_\infty} v$ (Langmuir). This model
 has to be complemented by suitable initial conditions for $u,v$ and boundary conditions on $\partial \Omega$ for $u$. The resulting model is often used in the literature for describing surfactant behavior, e.g.~\cite{Eggleton,Tasoglu,Quang,Chen2014}. The coefficients $\tilde k_{i,a}$, $\tilde k_{i,d}$ are the dimensionless adsorption and desorption coefficients.
\smallskip

\begin{remark} \rm
Sometimes, in the literature the Robin type interface conditions  $(-1)^{i} \nu_i \bn \cdot \nabla u_i  = \tilde k_{i,a} \tilde g(v) u_i- \tilde k_{i,d} v$ in \eqref{total} are  replaced by (simpler) Dirichlet type conditions. In case of ``instantaneous'' adsorption and desorption one may assume $\tilde k_{i,a} \gg \nu_i$, $\tilde k_{i,d} \gg \nu_i$ and the Robin interface conditions are approximated by $  \tilde k_{i,a}\tilde g(v) u_i= \tilde k_{i,d} v$, $i=1,2$.
\end{remark}
\smallskip

From a mathematical point of view, the problem \eqref{total} is challenging, because con\-vec\-tion-diffusion equations in the moving bulk phase $\Omega_i(t)$ are coupled with  a convection-diffusion equation on the moving interface $\Gamma(t)$. As far as we know, for this model there are no rigorous results on well-posedness known in the literature, cf. also  Remark~\ref{remlit} below.

\section{Simplified model} \label{S_simple}
As a first step in the analysis of the problem \eqref{total} we consider a simplified model. We restrict to $\tilde g(v)=v$ (Henry's law), consider  the \emph{stationary case} and make some (reasonable) assumptions on the range of the adsorption and desorption parameters $\tilde k_{i,a}$, $\tilde k_{i,d}$. In the remainder we assume that $\Omega_i$ and $\Gamma$ do not depend on $t$ (e.g.,   an equilibrium motion of a rising droplet in  a suitable frame of reference).
Since the interface is passively advected by the velocity field $\bw$, this
assumption leads to the constraint
\begin{equation}\label{w.n} \bw\cdot\bn=0   \quad \text{on}~~\Gamma.
 \end{equation}
 We also assume $\Div_\Gamma\bw=0$ so that the term $(\Div_\Gamma \bw) v$ in the surface convection-diffusion equation vanishes. Furthermore, for the spatial part of the material derivative $\dot v$  we have $\bw \cdot \nabla  v= \bw \cdot \unabla v$.  We let the normal part of $\bw$ to vanish on exterior boundary:
\begin{equation} \label{cond2}
 \bw \cdot \bn_{\Omega}=0\quad \text{on}~~\partial \Omega,
\end{equation}
where $\bn_{\Omega}$ denotes the outward pointing normal on $\partial \Omega$. 
From \eqref{cond2} it follows that there is no convective  mass flux across $\partial \Omega$. We also assume no diffusive mass flux across $\dO$, i.e.  the homogeneous Neumann boundary condition $ \bn_{\Omega} \cdot \nabla u_2 =0$ on $\partial \Omega$.
Restricting the model \eqref{total} to an  equilibrium state, we obtain the following stationary problem:
\begin{equation} \label{Total}
 \begin{aligned}
 - \nu_i\Delta u_i + \bw \cdot \nabla u_i  & = f_i \quad \text{in}~ \Omega_i, ~i=1,2, \\
 - \nu_\Gamma \Delta_\Gamma v + \bw \cdot \unabla v + K [ \nu \bn \cdot \nabla u]_\Gamma& = g  \quad \text{on}~~\Gamma, \\
  (-1)^{i} \nu_i \bn \cdot \nabla u_i & = \tilde k_{i,a}u_i- \tilde k_{i,d} v \quad \text{on}~~\Gamma, \quad i=1,2, \\
  \bn_{\Omega} \cdot \nabla u_2  & =0 \quad \text{on}~~\partial \Omega.
\end{aligned}
\end{equation}
In this model, we allow source terms $g\in L^2(\Gamma)$ and $f_i\in L^2(\Omega_i)$. Using partial integration over $\Omega_i$, $i=1,2$, and over $\Gamma$, one checks that these source terms have to satisfy the consistency condition
\begin{equation} \label{consfg}
  K\big( \int_{\Omega_1} f_1 \, d\bx + \int_{\Omega_2} f_2 \, d\bx\big) + \int_\Gamma g \, d\bs =0.
 \end{equation}
A simplified version of this model, namely with only \emph{one} bulk domain $\Omega_1$ and with $\bw=0$ (only diffusion) has recently been analyzed in \cite{ElliottRanner}.

From physics it is known that for surfactants almost always the desorption rates are (much) smaller than the adsorption rates. Therefore, it is reasonable to assume  $\tilde k_{i,d} \leq c \tilde k_{i,a}$ with a ``small'' constant $c$. To simplify the presentation, we assume $\tilde k_{i,d} \leq \tilde k_{1,a} + \tilde k_{2,a}$ for $i=1,2$.
For the  adsorption rates $\tilde k_{i,a}$ we exclude the singular-perturbed cases $\tilde k_{i,a} \downarrow 0$ and $\tilde k_{i,a} \to \infty$. Summarizing, we consider the parameter ranges
\begin{equation} \label{range}
  \tilde k_{i,a} \in [k_{\min},k_{\max}], \quad \tilde k_{i,d} \in [0,\tilde k_{1,a}+ \tilde k_{2,a}],
\end{equation}
with fixed generic constants $k_{\min} >0$, $k_{\max}$.  Note that unlike in previous studies we allow $\tilde k_{i,d}=0$ (i.e., only adsorption). Finally, due to the restriction on the adsorption parameter $\tilde k_{i,a}$ given in \eqref{range} we can use the following transformation to reduce the number of parameters of the model \eqref{Total}:
\begin{equation} \label{transfo}
\begin{split}
  & \tilde u_i := \tilde k_{i,a} u_i,  \quad \tilde v:=(\tilde k_{1,a}+ \tilde k_{2,a})v, \\
 &  \tilde \nu_i:= \tilde k_{i,a}^{-1}  \nu_i, \quad  \tilde \nu_\Gamma:=\nu_\Gamma/(\tilde k_{1,a}+ \tilde k_{2,a}) ,\\
 & \tilde  \bw_i:=  \tilde k_{i,a}^{-1} \bw \quad\text{in}~~\Omega_i, \quad \tilde \bw := \bw/(\tilde k_{1,a}+ \tilde k_{2,a}) \quad\text{on}~ \Gamma.
\end{split}
\end{equation}
Note that after this transformation $\tilde \bw$ will in general be  discontinuous across $\Gamma$.  In each subdomain and on $\Gamma$, however, $\tilde \bw$ is regular: $\tilde \bw_{|\Omega_i}=\tilde \bw_i \in H^{1,\infty}(\Omega_i)^3$, $i=1,2$, and $ \tilde \bw_{|\Gamma} \in H^1(\Gamma)^3$.
For simplicity we omit the tilde notation in the transformed variables. This then results in the following model, which we study in the remainder of the paper:
\begin{equation} \label{Totaltrans}
 \begin{aligned}
 - \nu_i\Delta u_i + \bw \cdot \nabla u_i  & = f_i \quad \text{in}~ \Omega_i, ~i=1,2, \\
 - \nu_\Gamma \Delta_\Gamma v + \bw \cdot \unabla v + K[ \nu \bn \cdot \nabla u]_\Gamma& = g  \quad \text{on}~~\Gamma, \\
  (-1)^{i} \nu_i \bn \cdot \nabla u_i & = u_i- q_i v \quad \text{on}~~\Gamma, \quad i=1,2, \\  \bn_{\Omega} \cdot \nabla u_2  & =0 \quad \text{on}~~\partial \Omega, \\ \text{with} ~~q_i & := \frac{\tilde k_{i,d}}{\tilde k_{1,a}+\tilde k_{2,a}} \in [0,1].
\end{aligned}
\end{equation}
The data $f_i$ and $g$ are assumed to satisfy the consistency condition \eqref{consfg}. Recall that $K= \frac{LU}{V}>0$, cf. \eqref{total},  is  a fixed (scaling) constant.

\section{Analysis of well-posedness} \label{S_analysis}
 In this section we derive a suitable weak formulation of the problem \eqref{Totaltrans} and prove well-posedness of this weak formulation. Concerning the smoothness of the interface, we assume that $\Gamma$ is a $C^1$ manifold. This assumption also suffices for the analysis in section~\ref{S_dual}.

We first introduce some  notations. For $u \in H^1(\Omega_1 \cup \Omega_2)$ we also write $u=(u_1,u_2)$ with $u_i= u_{|\Omega_i} \in H^1(\Omega_i)$. Furthermore:
\begin{align*}
(f,g)_\omega & := \int_{\omega}fg\,dx,~~~\|f\|_\omega^2:=(f,f)_\omega,\quad \text{where}~\omega~\text{is any of}~\{\Omega,\Omega_i,\Gamma\},\\
(\nabla u,\nabla w)_{\Omega_1 \cup \Omega_2} &:=\sum_{i=1,2}\int_{\Omega_i}\nabla u_i \cdot \nabla w_i\,dx, \quad u,w \in H^1(\Omega_1 \cup \Omega_2),\\
 \|u\|_{1, \Omega_1 \cup \Omega_2}^2 & := \|u_1\|_{H^1(\Omega_1)}^2 + \|u_2\|_{H^1(\Omega_2)}^2 = \|u\|_{\Omega}^2 + \|\nabla u\|_{\Omega_1 \cup \Omega_2}^2.
\end{align*}
We need a suitable gauge condition. In the original dimensional variables a natural condition is conservation of total mass, i.e. $(u_1,1)_{\Omega_1}+(u_2,1)_{\Omega_2}+(v,1)_\Gamma=m_0$, with $m_0 >0$ the initial total mass. Due to the transformation of variables and with an additional constant shift  this condition  is transformed to  $ K \tilde k_{1,a}^{-1}(u_1,1)_{\Omega_1}+K \tilde k_{2,a}^{-1}(u_2,1)_{\Omega_1} + (\tilde k_{1,a}+ \tilde k_{2,a})^{-1}(v,1)_\Gamma=0$ for the variables used in \eqref{Totaltrans}. Hence, we obtain the natural gauge condition
\begin{equation}\label{Gau}
 K (1+r)(u_1,1)_{\Omega_1}+ K(1+\frac{1}{r})(u_2,1)_{\Omega_2} + (v,1)_\Gamma=0, \quad r:= \frac{\tilde k_{2,a}}{\tilde k_{1,a}}.
\end{equation}
Define the product spaces
\begin{align*}
  \bV & =H^1(\Omega_1 \cup \Omega_2) \times H^1(\Gamma), \quad \|(u,v)\|_{\bV}:=\big( \|u\|_{1, \Omega_1 \cup \Omega_2}^2 +\|v\|_{1,\Gamma}^2\big)^\frac12, \\ \bVt & = \{\, (u,v) \in \bV~|~(u,v)~~\text{satisfies}~\eqref{Gau}\,\}.
\end{align*}

To obtain the weak formulation, we multiply the bulk and surface equation in \eqref{Totaltrans} by test functions from $\bV$, integrate by parts and use interface and boundary conditions. The resulting weak formulation
reads: Find $(u,v)\in \bVt$ such that for all $(\eta,\zeta)\in \bV$:
\begin{align}
 a((u,v);(\eta,\zeta)) & = (f_1,\eta_1)_{\Omega_1}+(f_2,\eta_2)_{\Omega_2} +(g,\zeta)_{\Gamma} ,\label{weak} \\
  a((u,v);(\eta,\zeta)) &:= (\nu\nabla u,\nabla \eta)_{\Omega_1 \cup \Omega_2} + (\bw \cdot \nabla u,\eta)_{\Omega_1\cup \Omega_2} + \nu_\Gamma(\unabla v,\unabla \zeta)_{\Gamma}  \nonumber \\  & + (\bw \cdot \unabla v,\zeta)_\Gamma
   +  \sum_{i=1}^2(u_i-q_iv, \eta_i- K \zeta)_\Gamma. \nonumber
\end{align}
For the further analysis, we note that both $\bw$-dependent parts of the bilinear form in \eqref{weak} are skew-symmetric:
\begin{equation}\label{w-forms}
(\bw \cdot \nabla u_i,\eta_i)_{\Omega_i}=-(\bw \cdot \nabla \eta_i,u_i)_{\Omega_i},~i=1,2,\quad
(\bw \cdot \unabla v,\zeta)_\Gamma=-(\bw \cdot \unabla \zeta,v)_\Gamma.
\end{equation}
To verify the first equality in \eqref{w-forms}, one integrates by parts over each subdomain $\Omega_i$:
\[ \begin{split}
(\bw \cdot \nabla u_1,\eta_1)_{\Omega_1} &  = -(\bw \cdot \nabla \eta_1,u_1)_{\Omega_1} -((\Div \bw)\eta_1,u_1)_{\Omega_1}+
((\bn\cdot\bw)\eta_1,u_1)_{\Gamma}, \\
(\bw \cdot \nabla u_2,\eta_2)_{\Omega_2} &= -(\bw \cdot \nabla \eta_2,u_2)_{\Omega_2} - ((\Div \bw)\eta_2,u_2)_{\Omega_2}  -
((\bn\cdot\bw)\eta_2,u_2)_{\Gamma} \\ & ~~~ +((\bn_{\Omega}\cdot\bw)\eta_2,u_2)_{\dO_2\cap\dO}.
\end{split} \]
All terms with $\Div \bw$, $\bn\cdot\bw$ or $\bn_{\Omega}\cdot\bw$ vanish   due to \eqref{div_w}, \eqref{w.n} and \eqref{cond2}.

The  variational formulation in \eqref{weak} is the basis for the finite element method introduced in section~\ref{S_FE}. For the analysis of well-posedness, it is convenient to introduce an \emph{equivalent} formulation where the test space $\bV$ is replaced by a smaller one, in which a suitable gauge condition is used. For this we define, for $\alpha=(\alpha_1,\alpha_2)$ with $\alpha_i \geq 0$, the space
\[
  \bV_\alpha:= \{\, (u,v) \in \bV~|~\alpha_1 (u_1,1)_{\Omega_1} + \alpha_2 (u_2,1)_{\Omega_2} +(v,1)_\Gamma=0 \, \}.
\]
Note that $\tilde \bV= \bV_{\alpha}$ for $\alpha=(K(1+r), K(1+\frac{1}{r}))$, cf.~\eqref{Gau}. The data $f_1,f_2,g$,  satisfy the consistency property \eqref{consfg}. From this and \eqref{w-forms} it follows that if a pair of trial and test functions $((u,v);(\eta,\zeta))$ satisfies \eqref{weak} then $((u,v);(\eta,\zeta)+\gamma (K,1))$ also satisfies \eqref{weak} for arbitrary $\gamma \in \mathbb{R}$. Now let an arbitrary $\alpha=(\alpha_1,\alpha_2)$ be given.  For every $(\eta,\zeta) \in \bV$ there exists $\gamma \in \mathbb{R}$ and $(\tilde \eta, \tilde \zeta) \in \bV_\alpha$ such that $(\eta,\zeta)=(\tilde \eta, \tilde \zeta) + \gamma (K,1)$ holds. From this it follows that \eqref{weak} is \emph{equivalent} to the following problem: Find $(u,v)\in \bVt$ such that for all $(\eta,\zeta)\in \bV_\alpha$:
\begin{equation}
 a((u,v);(\eta,\zeta))  = (f_1,\eta_1)_{\Omega_1}+(f_2,\eta_2)_{\Omega_2} +(g,\zeta)_{\Gamma} .\label{weak2}
\end{equation}
For this weak formulation we shall analyze well-posedness.

For $H^1(\Omega_i)$ and $H^1(\Gamma)$ the following Poincare-Friedrich's inequalities hold:
\begin{align}\label{Poincare}
\|u_i\|_{\Omega_i}^2 & \le c (\|\nabla u_i\|^2_{\Omega_i}+(u_i,1)^2_{\Omega_i}) \quad \text{for all}~~u_i \in H^1(\Omega_i),\\
\|u_i\|_{\Omega_i}^2 & \le c (\|\nabla u_i\|^2_{\Omega_i}+\|u_i\|^2_{\Gamma})\quad \text{for all}~~u_i \in H^1(\Omega_i), \label{Fr}\\
\|v\|_{\Gamma}^2  & \le c (\|\unabla v\|^2_{\Gamma}+(v,1)^2_{\Gamma})\quad \text{for all}~~v \in H^1(\Gamma). \label{Poincare2}
\end{align}
For the analysis of stability of the weak formulation we need the following Poincare type inequality in the space $\bV$.
\begin{lemma} \label{LempoincareA} Let $r_i,\sigma_i \in [0,\infty)$, $i=1,2$. There exists  $C_{P}(r_1,r_2,\sigma_1,\sigma_2) >0$ such that for all $(u,v) \in \bV$,  the following  inequality holds:
\begin{equation} \label{Poincar} \begin{split}
 & \|(u,v)\|_{\bV}  \le C_{P}\big(\|\nabla u\|_{\Omega_1 \cup \Omega_2} + \|\unabla v\|_{\Gamma}\\
  & + | r_1(u_1,1)_{\Omega_1}+r_2(u_2,1)_{\Omega_2}+(v,1)_{\Gamma}| +\sum_{i=1}^2|(u_i- \sigma_iv,1)_\Gamma|).
\end{split}
\end{equation}
\end{lemma}
\begin{proof}  The result follows from the Petree-Tartar Lemma (cf., \cite{ErnGuermond}). For convenience, we recall the lemma: Let $X,Y,Z$ be Banach spaces, $A \in L(X,Y)$ injective, $T\in L(X,Z)$ compact and assume
\begin{equation} \label{PT1}
  \|x\|_X \leq c \big( \|Ax\|_Y +\|Tx\|_Z\big) \quad \text{for all}~~x \in X.
\end{equation}
Then there exists a constant $c$ such that
\begin{equation} \label{PT2}
  \|x\|_X \leq c  \|Ax\|_Y  \quad \text{for all}~~x \in X
\end{equation}
holds.
We take $X=H^1(\Omega_1) \times H^1(\Omega_2) \times H^1(\Gamma)$ with the norm
\[
    \|(u_1,u_2,v)\|_X= (\|u_1\|_{1,\Omega_1}^2 +\|u_2\|_{1,\Omega_2}^2+ \|v\|_{1,\Gamma}^2)^{\frac12}.
\]
Furthermore, $Y=L^2(\Omega_1)^3 \times L^2(\Omega_2)^3 \times L^2(\Gamma)^3 \times \mathbb{R}^3$ with the product norm and $Z=L^2(\Omega_1) \times L^2(\Omega_2) \times L^2(\Gamma)$ with the product norm.
 We introduce the bilinear forms
\begin{align*}  l_0(u_1,u_2,v) & := r_1(u_1,1)_{\Omega_1} + r_2(u_2,1)_{\Omega_2} +(v,1)_{\Gamma}, \\ l_1(u_1,u_2,v)& =(u_1- \sigma_1 v,1)_\Gamma, \\ l_2(u_1,u_2,v) & =(u_2- \sigma_2 v,1)_\Gamma,
 \end{align*}
 and define the linear operators
\begin{align*}
 A(u_1,u_2,v) &= (\nabla u_1, \nabla u_2, \unabla v, l_0(u_1,u_2,v), l_1(u_1,u_2,v),l_2(u_1,u_2,v)),  \\
 T(u_1,u_2,v) &= (u_1,u_2,v). 
\end{align*}
Then we have $A \in L(X,Y)$. Consider $A(u_1,u_2,v)=0$. The first 9 equations yield $u_1=$constant, $u_2=$constant, $v=$constant, and substitution of this in the last three equations yields that these constants must be zero. Hence, $A$ is injective. The operator $T \in L(X,Z)$ is compact. This follows from the compactness of the embeddings $H^1(\Omega_i) \hookrightarrow L^2(\Omega)$, $H^1(\Gamma) \hookrightarrow L^2(\Gamma)$.
 It is easy to check that the inequality \eqref{PT1} is satisfied. The Petree-Tartar Lemma
  implies $(\|u\|_{1,\Omega_1 \cup \Omega_2}^2 + \|v\|_{1,\Gamma}^2)^\frac12 \leq c \|A(u_1,u_2,v)\|_Y$  and thus the estimate \eqref{Poincar} holds.\quad
\end{proof}
\medskip

The next theorem states an inf-sup stability estimate for the bilinear form in \eqref{weak2}.
\begin{theorem}\label{Infsup2}  There exists $C_{st}>0$ such that for all $q_1,q_2 \in [0,1]
$ and with a suitable $\alpha=\alpha(q_1,q_2)$ the following holds:
\begin{equation}\label{InfsuppA}
\inf_{(u,v)\in\bVt}\sup_{(\eta,\zeta)\in\bV_\alpha}\frac{a((u,v);(\eta,\zeta))}{\|(u,v)\|_{\bV}\|(\eta,\zeta)\|_{\bV}}
\ge C_{st}.
\end{equation}
\end{theorem}
\begin{proof} Let $(u,v)\in\bVt$ be given. Note that $(u,v)$ satisfies the gauge condition \eqref{Gau}. We first  treat the case $0 \leq q_2 \leq q_1 \leq 1$.
We consider three cases depending on values of these parameters.

We first consider $q_1,q_2 \in [0,\eps]$, with $\eps >0$ specified below.
We take $\eta_1=\beta u_1$, $\eta_2=\beta u_2$, with $\beta >0$, and $\zeta=v$. The value of $\beta$ is chosen further on. This yields
\begin{align*}
& a((u,v);(\eta,\zeta))  = \nu_1 \beta  \|\nabla u_1\|_{\Omega_1}^2 + \nu_2 \beta \|\nabla u_2\|_{\Omega_2}^2 + \nu_\Gamma  \|\unabla v \|_\Gamma^2 +\beta \|u_1\|_\Gamma^2 + \beta\|u_2\|_\Gamma^2 \\
 & -\sum_{i=1}^2(q_i \beta +K)(u_i,v)_\Gamma +K(q_1+q_2)\|v\|_\Gamma^2 \\
& \geq \beta \|\nu \nabla u\|_{\Omega_1\cup \Omega_2}^2  + \frac12 \beta \|u_1\|_\Gamma^2 + \frac12 \beta \|u_2\|_\Gamma^2  +\nu_\Gamma  \|\unabla v \|_\Gamma^2-(\eps \beta^\frac12 + K\beta^{-\frac12})^2 \|v\|_\Gamma^2 \\
 & \geq c_F \beta \|u\|_{1, \Omega_1 \cup \Omega_2}^2 + \nu_\Gamma  \|\unabla v \|_\Gamma^2-(\eps \beta^\frac12 + K\beta^{-\frac12})^2 \|v\|_\Gamma^2,
\end{align*}
where in the last inequality we used \eqref{Fr}. The constant $c_F >0$ depends only on the Friedrich's constant from \eqref{Fr} and the viscosity $\nu$. From the gauge condition we get
\[
  (v,1)^2 \leq 2 K^2\big( (1+r)^2 (u_1,1)_{\Omega_1}^2 +(1+\frac{1}{r})^2 (u_2,1)_{\Omega_2}^2\big) \leq c \big(\|u_1\|_{\Omega_1}^2 + \|u_2\|_{\Omega_2}^2 \big)=c\|u\|^2_\Omega.
\]
Using this and the Poincare's inequality in \eqref{Poincare2} we obtain
\begin{align*}
 & a((u,v);(\eta,\zeta)) \\ &   \geq c_F \beta \|u\|_{1, \Omega_1 \cup \Omega_2}^2 +\nu_\Gamma ( \|\unabla v \|_\Gamma^2 + (v,1)^2\big)  - \hat c  \|u\|^2_\Omega
  -(\eps \beta^\frac12 + K\beta^{-\frac12})^2 \|v\|_\Gamma^2  \\ & \geq  c_F \beta \|u\|_{1, \Omega_1 \cup \Omega_2}^2 + \hat c_F \|v\|_{1,\Gamma}^2 - \hat c  \|u\|^2_\Omega- (\eps \beta^\frac12 + K\beta^{-\frac12})^2 \|v\|_\Gamma^2.
\end{align*}
The constant $\hat c$ depends only on $\nu_\Gamma$, $r,K$.
The constant $\hat c_F >0$ depends only on a Poincare's constant and $\nu_\Gamma$. We take $\beta$ sufficiently large (depending only on $c_F$, $\hat c$ and $K$) and $\eps >0$ sufficiently small such that the third term can be adsorbed in the first one and the last term can be adsorbed in the second one. Thus we get
\[
a((u,v);(\eta,\zeta))  \geq c \|(u,v)\|_\bV^2 \geq c  \|(u,v)\|_\bV  \|(\eta,\zeta)\|_\bV,
\]
which  completes the proof of \eqref{InfsuppA} for the first case. Now $\eps >0$ is fixed.

In the second case we take $q_1\geq \eps$, and $q_2 \in [0,\delta]$, with a $\delta \in (0,\eps]$ that will be specified below. We take $\eta_1=u_1, \eta_2=0, \zeta = K^{-1}q_1 v$. Using the gauge condition and \eqref{Poincar} with $u_2=0$, $\sigma_2=0$, $r_1=K(1+r)$, $\sigma_1=q_1$, we get
\begin{align}
  a((u,v);(\eta,\zeta))& = \nu_1  \|\nabla u_1\|_{\Omega_1}^2 + \frac{q_1 \nu_\Gamma}{K}  \|\unabla v \|_\Gamma^2 +\|u_1-q_1v\|_\Gamma^2 - q_1(u_2,v)_\Gamma + q_2 q_1 \|v\|_{\Gamma}^2 \nonumber \\
 & \geq \nu_1  \|\nabla u_1\|_{\Omega_1}^2 + \frac{\eps \nu_\Gamma}{K}  \|\unabla v \|_\Gamma^2 +\|u_1-q_1v\|_\Gamma^2  \nonumber \\ & +|K(1+r)(u_1,1)_{\Omega_1}+(v,1)_\Gamma|^2   - K^2(1+\frac{1}{r})^2(u_2,1)_{\Omega_2}^2 -q_1(u_2,v)_\Gamma \nonumber \\
 & \geq c_F \big( \|u_1\|_{1,\Omega_1}^2 + \|v\|_{1, \Gamma}^2\big) - c \|u_2\|_{\Omega_2}^2   -\|u_2\|_\Gamma \|v\|_\Gamma \nonumber \\
 & \geq \frac12 c_F \big( \|u_1\|_{1,\Omega_1}^2 + \|v\|_{1, \Gamma}^2\big) - c\big( \|u_2\|_{\Omega_2}^2+ \|u_2\|_\Gamma^2\big). \label{pp}
\end{align}
The constant $c_F >0$ depends on Poincare's constant and on $\eps$. We now take $\eta_1=0, \eta_2=\beta u_2$, with $\beta >0$ and  $\zeta =0$. This yields, cf. \eqref{Fr},
\[
 a((u,v);(\eta,\zeta)) = \nu_2  \beta \|\nabla u_2\|_{\Omega_1}^2 +  \beta \|u_2\|_\Gamma^2 - \beta q_2(u_2,v)_\Gamma  \geq c \beta \|u_2\|_{1,\Omega_2}^2 - \frac12  \beta \delta^2 \|v\|_{\Gamma}^2 .
\]
Combining this with \eqref{pp} and taking $\beta$ sufficiently large such that the last term in \eqref{pp} can be adsorbed, we obtain for $\eta_1=u_1, \eta_2=\beta u_2, \zeta = K^{-1}q_1 v$:
\begin{align*}
  a((u,v);(\eta,\zeta)) \geq  \frac12 c_F\big( \|u_1\|_{1,\Omega_1}^2 + \|v\|_{1, \Gamma}^2\big) + c \beta \|u_2\|_{1,\Omega_2}^2 - \frac12  \beta \delta^2 \|v\|_{\Gamma}^2 .
\end{align*}
Now we take $\delta >0$ sufficiently small such that the last term can be adsorbed by the second one. Hence,
\[
a((u,v);(\eta,\zeta))  \geq c \|(u,v)\|_\bV^2 \geq c  \|(u,v)\|_\bV  \|(\eta,\zeta)\|_\bV,
\]
which completes the proof of the inf-sup property for the second case.
Now $\delta >0$ is fixed.

We consider the last case, namely $ q_1 \geq \delta$ and $q_2 \geq \delta$.
Take $\eta_1=u_1$, $\eta_2=\frac{q_1}{q_2} u_2$, $\zeta= K^{-1} q_1 v$. We then get
\begin{align*}
a((u,v);(\eta,\zeta)) & = \nu_1 \|\nabla u_1\|_{\Omega_1}^2 + \nu_2 \frac{q_1}{q_2}\|\nabla u_2\|_{\Omega_2}^2 + \frac{\nu_\Gamma q_1}{K} \|\unabla v \|_\Gamma^2 \\
  & + \|u_1 - q_1 v\|_\Gamma^2 + \frac{q_1}{q_2}\|u_2 - q_2 v\|_\Gamma^2 \\
 & \geq c \big(\|\nabla u\|_{\Omega_1 \cup \Omega_2}^2 +\|\unabla v \|_\Gamma^2  + \sum_{i=1}^2 \|u_i - q_i v\|_\Gamma^2  \big).
\end{align*}
We use \eqref{Poincar} with $r_1=K(1+r)$, $r_2=K(1+\frac{1}{r})$, with $r$ from \eqref{Gau},  and $\sigma_i=q_i$. This yields
\[
a((u,v);(\eta,\zeta))  \geq c \|(u,v)\|_\bV^2 \geq c  \|(u,v)\|_\bV  \|(\eta,\zeta)\|_\bV,
\]
with a constant $c>0$ that depends on $\delta$, but is independent of $(u,v)$.

In all three cases, since $(u,v)$ obeys  the gauge condition \eqref{Gau}, we get $(\eta,\zeta) \in V_\alpha$, for suitable $\alpha=(\alpha_1,\alpha_2)$ with $\alpha_i >0$.\\
The case $0 \leq q_1 \leq q_2 \leq 1$ can be treated with almost exactly the same arguments.
\end{proof}
\medskip

Note that the $\alpha$ used in Theorem~\ref{Infsup2} may depend on $q_i$. In the remainder, for given problem parameters $q_i \in [0,1]$, $i=1,2$, we take $\alpha$ as in Theorem~\ref{Infsup2} and use this  $\alpha$ in  the weak formulation \eqref{weak}. For the analysis of a dual problem, we also need the stability of the adjoint bilinear form
given in the next lemma.
\begin{lemma} \label{lemm}
There exists $C_{st}>0$ such that for all $q_1,q_2 \in [0,1]
$ and with  $\alpha$ as in Theorem~\ref{Infsup2} the following holds:
\begin{equation}\label{Infsupp}
\inf_{(\eta,\zeta)\in\bV_\alpha}\sup_{(u,v)\in\bVt}\frac{a((u,v);(\eta,\zeta))}{\|(u,v)\|_{\bV}\|(\eta,\zeta)\|_{\bV}}
\ge C_{st}.
\end{equation}
 \end{lemma}
\begin{proof}
 Take  $(\eta,\zeta) \in \bV_\alpha$, $(\eta,\zeta) \neq (0,0)$.
 The arguments of the proof of Theorem~\ref{Infsup2}  show that for $(\eta,\zeta) \in \bV_\alpha $
  there exists $(u,v) \in \bVt$ such that $a((u,v);(\eta,\zeta)) \ge C_{st}\|(u,v)\|_{\bV}\|(\eta,\zeta)\|_{\bV}$ holds, with the same constant as \eqref{InfsuppA}.
\end{proof}
\medskip

Finally, we give a result on continuity of the bilinear form.
\begin{lemma} \label{lemmA}
 There exists a constant $c$ such that for all $q_1,q_2 \in [0,1]$ the following holds:
\[
 a((u,v);(\eta,\zeta)) \leq c \|(u,v)\|_\bV \|(\eta,\zeta)\|_\bV \quad \text{for all}~~(u,v), \, (\eta,\zeta) \in \bV.
\]
\end{lemma}
\begin{proof}
 The continuity estimate is a direct consequence of Cauchy-Schwarz inequalities and boundedness of the trace operator.\quad
\end{proof}
\medskip

We obtain the following well-posedness and regularity  results.
\begin{theorem}\label{Th_wp} For any $f_i\in L^2(\Omega_i)$, $i=1,2$, $g\in L^2(\Gamma)$ such that \eqref{consfg} holds, there exists a unique
solution $(u,v) \in \bVt$ of \eqref{weak}, which is also the unique solution to \eqref{weak2}. This solution satisfies the a-priori estimate
\begin{equation}\label{apr_est1}
\|(u,v)\|_{\bV}\le C \|(f_1,f_2,g)\|_{\bV'}\le c (\|f_1\|_{\Omega_1} +\|f_2\|_{\Omega_2} +\|g\|_{\Gamma}),
\end{equation}
with  constants $C,c$ independent of $f_i$, $g$ and $q_1,q_2 \in [0,1]$. If in addition $\Gamma$ is a  $C^2$-manifold and $\Omega$ is convex or $\partial \Omega$ is $C^2$ smooth, then $u_i\in H^2(\Omega_i)$, for $i=1,2$, and
$v\in H^2(\Gamma)$. Furthermore, the solution satisfies the second a-priori  estimate
\begin{equation}\label{apr_est2}
\|u_1\|_{H^2(\Omega_1)}+\|u_2\|_{H^2(\Omega_2)} + \|v\|_{H^2(\Gamma)} \le c (\|f_1\|_{\Omega_1} +\|f_2\|_{\Omega_2}+\|g\|_{\Gamma}).
\end{equation}
\end{theorem}
\begin{proof}
Existence, uniqueness and the first a-priori estimate follow from Theorem~\ref{Infsup2} and the Lemmas~\ref{lemm} and \ref{lemmA}.
To show the extra regularity of the solution, we note that $u_i$ satisfies the weak formulation of the Poisson equation $-\nu_i\Delta u_i=F_i:=f_i-\bw\cdot\nabla u_i$ in $\Omega_i$   with a Robin boundary condition
$(-1)^i\bn\cdot \nabla u_i-u_i=G_i:=-q_i v$. Thanks to \eqref{apr_est1} we have $F_i\in L^2(\Omega_i)$, $G_i\in H^{\frac12}(\dO_i)$ and $\|F_i\|_{\Omega_i}+\|G_i\|_{H^{\frac12}(\dO_i)}\le
c (\|f_1\|_{\Omega_1} +\|f_2\|_{\Omega_2}+\|g\|_{\Gamma})$. Theorem~2.4.2.6 from~\cite{Grisvard} implies  $u_i\in H^2(\Omega_i)$ and
the estimate for $u_i$ in \eqref{apr_est2}. On the interface $\Gamma$, $v$ satisfies the weak formulation of the Laplace-Beltrami equation
$-\nu_\Gamma\Delta_\Gamma v= G_\Gamma:=g-\bw\cdot\nabla_{\Gamma} v-K[\bn\cdot \nabla u]$. Thanks to \eqref{apr_est1}, the regularity result for $u_i$ in \eqref{apr_est2} and the smoothness of $\Gamma$, we have $G_\Gamma\in L^2(\Gamma)$ and $\|G_\Gamma\|_\Gamma\le c (\|f_1\|_{\Omega_1} +\|f_2\|_{\Omega_2}+\|g\|_{\Gamma})$.
Now we apply the regularity result for the Laplace-Beltrami equation on a closed $C^2$-surface from Lemma~3.2 in~\cite{DziukActaNumerica}.
This proves $v\in H^2(\Gamma)$ and the estimate on $\|v\|_{H^2(\Gamma)}$ in \eqref{apr_est2}.
\end{proof}
\ \\
\begin{remark} \label{remlit}
 \rm In \cite{BurmanHansbo,ElliottRanner} a stationary diffusion problem on a bulk domain is linearly coupled with a stationary diffusion equation on the boundary of this domain. Hence there is only \emph{one} bulk domain. Well-posedness of a suitable weak formulation of this problem is shown in \cite{ElliottRanner,Alphonso2014}. The analysis in \cite{ElliottRanner} is significantly simpler than the one presented above. This is due to the fact that for the case of one bulk domain the coupling term is simpler and one easily verifies that the corresponding bilinear form is elliptic. In our case, due to the coupling term  $\sum_{i=1}^2(u_i-q_iv, \eta_i- K \zeta)_\Gamma$ in \eqref{weak}, the bilinear form is not elliptic and we have to derive an inf-sup estimate. A further complication, compared to the case of one bulk domain,   is the Poincare type inequality that we need, cf. Lemma~\ref{LempoincareA}.
\end{remark}

\section{Adjoint problem}\label{S_dual}
Consider the following formal adjoint problem, with   $\alpha$ as in Theorem~\ref{Infsup2}. For given $f\in L^2(\Omega)$ and $g\in L^2(\Gamma)$
find $(u,v)\in \bV_\alpha$ such that for all $(\eta,\zeta)\in \bVt$:
\begin{equation}
 a((\eta,\zeta);(u,v))  = (f_1,\eta_1)_{\Omega_1}+(f_2,\eta_2)_{\Omega_2} +(g,\zeta)_{\Gamma}.\label{weak_adj}
\end{equation}
Due to the results in Theorem~\ref{Infsup2} and Lemmas~\ref{lemm}, \ref{lemmA}, the problem \eqref{weak_adj} is well-posed.

Now we look for the corresponding strong formulation of this adjoint problem. We introduce an
appropriate gauge condition for the right-hand side:
\begin{equation} \label{rhs_gaude_Adjoint}
q_1(f_1,1)_{\Omega_1}+q_2(f_2,1)_{\Omega_2}+(g,1)_\Gamma=0.
\end{equation}
For any $(\eta_1,\eta_2,\zeta)\in \bV$ there is a $\gamma \in \mathbb{R}$ such that $(\eta_1,\eta_2,\zeta) +\gamma (q_1,q_2,1)\in \bVt$ holds. From the definition of the bilinear form it follows that   $a((q_1,q_2,1);(u,v))=0$ holds. Hence, if the right-hand side satisfies condition \eqref{rhs_gaude_Adjoint}, the formulation \eqref{weak_adj} is equivalent to:  Find $(u,v)\in \bV_\alpha$ such that
\begin{equation*}
 a((\eta,\zeta);(u,v))  = (f_1,\eta_1)_{\Omega_1}+(f_2,\eta_2)_{\Omega_2} +(g,\zeta)_{\Gamma}
 \quad \text{for all}~~(\eta,\zeta)\in \bV.
\end{equation*}
Varying $(\eta,\zeta)$ we find the strong formulation of the dual problem to \eqref{Totaltrans}:
\begin{equation} \label{Adjoint}
 \begin{aligned}
 - \nu_i\Delta u_i - \bw \cdot \nabla u_i  & = f_i \quad \text{in}~ \Omega_i, ~i=1,2, \\
 - \nu_\Gamma \Delta_\Gamma v - \bw \cdot \unabla v + [ q\nu \bn \cdot \nabla u]_\Gamma& = g  \quad \text{on}~~\Gamma, \\
  (-1)^{i} \nu_i \bn \cdot \nabla u_i & = u_i- K v \quad \text{on}~~\Gamma, \quad i=1,2, \\  \bn_{\Omega} \cdot \nabla u_2  & =0 \quad \text{on}~~\partial \Omega.
\end{aligned}
\end{equation}
In the second equation $q$ denotes a piecewise constant function with values $q_{|\Omega_i}:=q_i$.
Note that compared to the original primal problem \eqref{Totaltrans} we now have $-\bw$ instead of $\bw$ and that the roles of $K$ and $q$ are interchanged.
With the same arguments as for the primal problem, cf. Theorem~\ref{Th_wp}, the following $H^2$-regularity result for the dual problem can be derived.
\begin{theorem}\label{Th_wpd} For any $f_i\in L^2(\Omega_i)$, $i=1,2$, $g\in L^2(\Gamma)$ such that \eqref{rhs_gaude_Adjoint} holds, there exists a unique weak
solution $(u,v) \in \bV_\alpha$ of \eqref{Adjoint}. If $\Gamma$ is a  $C^2$-manifold and $\Omega$ is convex or $\partial \Omega$ is $C^2$ smooth, then $u_i\in H^2(\Omega_i)$, for $i=1,2$, and
$v\in H^2(\Gamma)$ satisfy the  a-priori  estimate
\[
\|u_1\|_{H^2(\Omega_1)}+\|u_2\|_{H^2(\Omega_2)} + \|v\|_{H^2(\Gamma)} \le c (\|f_1\|_{\Omega_1} +\|f_2\|_{\Omega_2}+\|g\|_{\Gamma}).
\]
\end{theorem}

\section{Unfitted finite element method}\label{S_FE}
Let the domain $\Omega\subset\mathbb{R}^3$ be  polyhedral
and  $\{\T_h\}_{h>0}$   a family of tetrahedral triangulations of $\Omega$ such that $\max\limits_{T\in\T_h}\mbox{diam}(T) \le h$.
These triangulations are assumed to be regular, consistent and stable.

It is computationally convenient to allow triangulations that are \emph{not fitted}  to the interface $\Gamma$. We use
a `discrete' interface $\Gamma_h$, which approximates $\Gamma$ (as specified below).
To this end, assume that the surface $\Gamma$ is implicitly defined as the zero set of a non-degenerate  level set function $\phi$:
\[
\Gamma=\{\bx\in\Omega \,:\,\phi(\bx)=0\},\]
where $\phi$ is $C^1$ smooth function in a neighborhood of $\Gamma$, such that
\begin{equation} \label{propphi}
\phi<0~~\text{in}~\Omega_1,\quad \phi>0~~\text{in}~\Omega_2,\quad\text{and}\quad
|\nabla \phi|\ge c_0>0~\text{in}~U_\delta\subset\Omega.
\end{equation}
Here $U_\delta\subset\Omega$ is a tubular neighborhood of $\Gamma$ of width $\delta$: $U_\delta=\{\bx \in \mathbb{R}^3: {\rm dist}(\bx,\Gamma) < \delta\}$, with $\delta >0$ a  sufficiently small constant.  A special choice for  $\phi$ is the signed distance function to $\Gamma$.
Let  $\phi_h$ be a given continuous piecewise polynomial approximation (w.r.t. $\T_h$) of the level set function $\phi$ which satisfies
\begin{equation}\label{phi_h}
\|\phi-\phi_h\|_{L^\infty(U_\delta)}+ h \|\nabla(\phi-\phi_h)\|_{L^\infty(U_\delta)}\le c\,h^{q+1},
\end{equation}
with some $q\ge1$. For this estimate to hold, we assume that the level set function $\phi$ has the smoothness property $\phi \in C^{q+1}(U_\delta)$. Clearly this induces a similar smoothness property for its zero level $\Gamma$. Then we define
\begin{equation}
\Gamma_h:=\{\, \bx \in  \Omega~:~\phi_h(\bx)=0\,\},\label{Omega_h}
\end{equation}
and assume that $h$ is sufficiently small such that $\Gamma_h \subset U_\delta$ holds. Furthermore
\begin{equation}
\begin{aligned}
\Omega_{1,h}&:= \{\, \bx \in  \Omega~~:~\phi_h(\bx) < 0\, \},\\
 \Omega_{2,h}&:= \{\, \bx \in  \Omega~~:~\phi_h(\bx) > 0\, \}.
 \end{aligned}
\end{equation}

From \eqref{propphi} and \eqref{phi_h} it follows that
\begin{equation} \label{Dist}
{\rm dist}(\Gamma_h , \Gamma)=\max_{\bx\in \Gamma_h} {\rm dist}(\bx , \Gamma) \leq c h^{q+1}
 \end{equation}
holds.
In many applications  only such a finite element  approximation $\phi_h$ (e.g, resulting from the level set method) to the level set $\phi$ is known. For such a situation the finite element method formulated below is particularly well suited.
In cases where $\phi$ is known, one can take $\phi_h:=I_h(\phi)$, where  $I_h$ is a suitable
piecewise polynomial interpolation operator. If $\phi_h$ is a $P_1$ continuous finite element function, then
$\Gamma_h$ is a piecewise planar closed surface. In this practically convenient case, it is reasonable (if $\phi \in C^2(U_\delta)$) to assume that \eqref{phi_h} holds with $q=1$.

Consider the space of all continuous piecewise polynomial functions of  degree $k\ge1$ with respect to  $\T_h$:
\begin{equation} \label{bulk}
V_h^{\rm bulk}:=\{v\in C(\Omega)\,:~ v|_T\in P_k(T)\quad\forall\,T\in\T_h\}.
\end{equation}
We now define three \emph{trace spaces} of finite element functions:
\begin{equation}\label{FEspace}
\begin{aligned}
V_{\Gamma,h}&:=\{v\in C(\Gamma_h)~~:~ v=w|_{\Gamma_h}~~\text{for some}~w\in V_h^{\rm bulk}\},\\
V_{1,h}&:=\{v\in C(\Omega_{1,h})\,:~ v=w|_{\Omega_{1,h}}~~\text{for some}~w\in V_h^{\rm bulk}\},\\
V_{2,h}&:=\{v\in C(\Omega_{2,h})\,:~ v=w|_{\Omega_{2,h}}~~\text{for some}~w\in V_h^{\rm bulk}\}.
 \end{aligned}
\end{equation}
We need the  spaces $V_{\Omega,h}=V_{1,h}\times V_{2,h}$ and $\bV_{h}= V_{\Omega,h}\times V_{\Gamma,h} \subset H^1(\Omega_{1,h} \cup \Omega_{2,h}) \times H^1(\Gamma_h)$. The space $V_{\Omega,h}$ is studied in many papers on the so-called  cut finite element method or XFEM \cite{Hansbo02,Hansbo04,Belytschko03,Fries}. The trace space $V_{\Gamma,h}$  is introduced in \cite{OlshReusken08}. For representation of functions in these trace spaces we use the standard nodal basis functions in the space $V_h^{\rm bulk}$. In the space $V_{\Gamma,h}$ these functions do not yield a basis, but form only a frame. For the case $k=1$ this linear algebra issue is studied in \cite{OlshanskiiReusken08}.

We consider the finite element bilinear form on $\bV_{h}\times\bV_{h}$, which results from the bilinear form of
the differential problem using integration by parts in advection terms and further replacing $\Omega_{i}$ by $\Omega_{i,h}$ and
$\Gamma$ by $\Gamma_h$:
\[
 \begin{split}
a_h((u,v);(\eta,\zeta)) &= \sum_{i=1}^2\left\{(\nu_i\nabla u,\nabla \eta)_{\Omega_{i,h}} + \frac12\left[(\bw_h \cdot \nabla u,\eta)_{\Omega_{i,h}}-(\bw_h \cdot \nabla \eta, u)_{\Omega_{i,h}}\right]\right\}\\  & + \nu_\Gamma(\unablah v,\unablah \zeta)_{\Gamma_h} + \frac12\left[(\bw_h \cdot \unablah v,\zeta)_{\Gamma_h}-(\bw_h \cdot \unablah \zeta,v)_{\Gamma_h}\right]\\  & + \sum_{i=1}^2(u_i- q_i v,\eta_i-K\zeta)_{\Gamma_h}.
\end{split}
\]
In this formulation we use the transformed quantities as in \eqref{transfo}, but with $\Omega_i$, $\Gamma$ replaced by $\Omega_{i,h}$ and $\Gamma_h$, respectively. For example, on $\Omega_{i,h}$ we use the transformed viscosity $\tilde \nu_i:=\tilde k_{i,a}^{-1} \nu_i$, with $\nu_i$ the dimensionless  viscosity as in \eqref{total}. Similarly, the transformed velocity field $\tilde \bw_h \in [H^1(\Omega_{1,h}\cup\Omega_{2,h})]^3 \cap H^1(\Gamma_h)^3$ is obtained after the transformation $ \tilde \bw_h := \tilde k_{i,a}^{-1} \bw$ on $ \Omega_{i,h}$, $i=1,2$,  and $\tilde \bw_h:= \bw/(\tilde k_{1,a}+ \tilde k_{2,a})$ on $\Gamma_h$,  with $\bw$ the dimensionless smooth velocity vector field as in \eqref{total}.
As in \eqref{Totaltrans} we omit the tilde notation in the transformed quantities. In the derivation  of the skew-symmetry result \eqref{w-forms} we used that $\bw\cdot \bn =0$ holds on $\Gamma$. The property $\bw_h\cdot \bn_h =0$, however, does not necessarily hold on $\Gamma_h$. Skew-symmetry of the convection terms in $a_h(\cdot,\cdot)$ is enforced by using the skew-symmetric forms in the square brackets above.
Let $g_h\in L^2(\Gamma_h)$, $f_h\in L^2(\Omega)$ be given and satisfy
\begin{equation}\label{gaugeh}
K(f_h,1)_{\Omega}+(g_h,1)_{\Gamma_h}=0.
\end{equation}
As discrete gauge condition we introduce, cf. \eqref{Gau},
\begin{equation} \label{Gaugediscr}
 K(1+r)(u_h,1)_{\Omega_{1,h}}+ K(1+\frac{1}{r})(u_h,1)_{\Omega_{2,h}} + (v_h,1)_{\Gamma_h}=0, \quad r:= \frac{\tilde k_{2,a}}{\tilde k_{1,a}}.
\end{equation}
Furthermore, define
\[
\bV_{h,\alpha}:=\{\, (\eta,\zeta)\in \bV_h \,:\,\alpha_1 (\eta,1)_{\Omega_{1,h}} + \alpha_2 (\eta,1)_{\Omega_{2,h}} + (\zeta,1)_{\Gamma_h}=0\},\quad
\]
for arbitrary (but fixed) $\alpha_1, \alpha_2\ge 0$, and $\bVt_{h}:=\bV_{h,\alpha},$ with $\alpha_1=K(1+r)$, $\alpha_2=K(1+\frac{1}{r})$.
The finite element method is as follows: Find $(u_h,v_h)\in \bVt_h$ such that
\begin{equation} \label{weakh}
a_h((u_h,v_h);(\eta,\zeta))=(f_h,\eta)_{\Omega}+(g_h,\zeta)_{\Gamma_h}\quad\text{for all}~(\eta,\zeta)\in\bV_h.
\end{equation}
With the same arguments as for the continuous problem, cf.~\eqref{weak2}, based on the consistency condition \eqref{gaugeh} we obtain an equivalent discrete problem if the test space $\bV_h$ is replaced by $\bV_{h,\alpha}$. The latter formulation is used in the analysis below.
We shall use the Poincare and Friedrich's inequalities \eqref{Poincare}--\eqref{Poincar} with  $\Omega_{i}$ replaced by $\Omega_{i,h}$ and $\Gamma$ by $\Gamma_h$. We assume that the corresponding Poincare-Friedrich's constants are bounded uniformly in $h$.

In the finite element space we use the norm given by
\[ \|(\eta,\zeta)\|_{\bV_h}^2:= \|\eta\|_{H^1(\Omega_{1,h} \cup \Omega_{2,h})}^2+ \|\zeta\|_{H^1(\Gamma_h)}^2, \quad (\eta,\zeta) \in H^1(\Omega_{1,h} \cup \Omega_{2,h}) \times H^1(\Gamma_h).
\]
Repeating the same arguments as in the proof of Theorem~\ref{Infsupp} and Lemmas~\ref{lemm}, and~\ref{lemmA}, we obtain an inf-sup stability result for the discrete
bilinear form and its dual as well as a continuity estimate.

\begin{theorem}\label{L_infsup2h}  {\rm(i)}~~For any  $q_1,q_2 \in [0,1]$, there exists  $\alpha$  such that
\begin{equation}\label{infsup2h}
\inf_{(u,v)\in\bVt_h}\sup_{(\eta,\zeta)\in\bV_{h,\alpha}}\frac{a_h((u,v);(\eta,\zeta))}{\|(u,v)\|_{\bV_h}\|(\eta,\zeta)\|_{\bV_h}}
\ge C_{st}>0,
\end{equation}
with a positive constant $C_{st}$ independent of $h$ and of $q_1,q_2 \in [0,1]$.\\
{\rm(ii)} There is a constant $c$ independent of $h$ such that
\begin{equation}\label{conth}
a_h((u,v);(\eta,\zeta))\le c\|(u,v)\|_{\bV_h}\|(\eta,\zeta)\|_{\bV_h }
\end{equation}
for all $(u,v), (\eta,\zeta)\in H^1(\Omega_{1,h}\cup \Omega_{2,h})\times H^1(\Gamma_h)$.
\end{theorem}

As a corollary of this theorem we obtain  the well-posedness result for the discrete problem.
\begin{theorem}\label{Th_wph} For any $f_h\in L^2(\Omega_h)$, $g_h\in L^2(\Gamma_h)$ such that \eqref{gaugeh} holds, there exists a unique solution $(u_h,v_h)\in \bV_h$ of \eqref{weakh}. For this solution the a-priori estimate
\[
\|(u_h,v_h)\|_{\bV_h}\le C_{st}^{-1} \|(f_h,g_h)\|_{\bV'_h}\le c (\|f_h\|_{\Omega_h}+\|g_h\|_{\Gamma_h})
\]
holds. The constants $C_{st}$ and $c$ are independent of $h$.
\end{theorem}
\section{Error analysis}\label{s_error} For the error analysis we assume that the level set function, which characterizes the interface $\Gamma$, has smoothness $\phi \in C^{q+1}(U_\delta)$, with $q \geq 1$. This implies that the estimate \eqref{phi_h} holds. In the analyis below, we also need the smoothness assumption $\Gamma\in C^{k+1}$, where $k \geq 1$ is the degree of the polynomials used in the finite element space, cf.~\eqref{bulk}.  Concerning the velocity field $\bw$ we need that the original (unscaled) velocity $\bw$ is sufficiently smooth, $\bw \in H^{1, \infty}(\Omega)$. In the remainder of this section and in section~\ref{S_errorL2} we assume that these smoothness requirements are satisfied. \\
The smoothness properties of $\Gamma$ imply that there exists a $C^2$ signed distance function $d:U_\delta \rightarrow \mathbb{R}$ such that $\Gamma=\{\bx \in U_\delta: d(\bx)=0\}$.     We assume that $d$ is negative on $\Omega_1 \cap U_\delta$ and positive on $\Omega_2 \cap U_\delta$.
Thus for $\bx \in U_\delta$, ${\rm dist}(\bx, \Gamma)=|d(\bx)|$.
Under these conditions, for $\delta >0 $ sufficiently small, but independent of $h$, there is an orthogonal projection $p:U_\delta \rightarrow \Gamma$ given by $\bp(\bx)=\bx-d(\bx) \bn(\bx)$, where $\bn(\bx)=\nabla d(\bx)$.
Let $\bH=D^2 d= \nabla \bn$ be the Weingarten map. More details of the present formalism can be found in  \cite{DD07}, \S~2.1.

Given $v \in H^1(\Gamma)$, we denote by  $v^e \in H^1(U_\delta)$ its extension from $\Gamma$ along normals, i.e. the function defined by $v^e(\bx)=v(\bp(\bx))$; $v^e$ is constant in the direction normal to $\Gamma$.
The following holds:
\begin{equation}
\nabla  v^e(\bx)  = (\mathbf{I} -  d(\bx)\bH(\bx))\unabla v(\bp(\bx))\quad \text{for}~~\bx \in U_\delta.
 \label{nabla1}
\end{equation}
We need some further (mild)  assumption on how well the mesh resolves the geometry of the (discrete) interface.
We assume that $\Gamma_h\subset U_\delta$  is the graph of a function $\gamma_h(\bs)$, $\bs\in\Gamma$  in the local coordinate system $(
\bs,r)$, $\bs \in \Gamma$, $r \in [-\delta,\delta]$, with  $\bx=\bs
+ r \bn(\bs)$:
\[ \Gamma_h=\{\, (\bs, \gamma_h(\bs))~: ~\bs \in \Gamma\,\}.
 \]
From \eqref{Dist} it follows that
\begin{equation} \label{Dist1}
 |\gamma_h(\bs)| ={\rm dist}\big(\bs + \gamma_h(\bs) \bn(\bs),\Gamma\big) \leq {\rm dist}(\Gamma_h, \Gamma) \leq c h^{q+1},
\end{equation}
with a constant $c$ independent of $\bs \in \Gamma$.

\subsection{Bijective mapping $\Omega_{i,h} \to \Omega_i$}
For the analysis of the consistency error we need a bijective mapping  $\Omega_{i,h} \to \Omega_i$, $i=1,2$. We use a mapping that is similar to the one given in Lemma~5.1 in \cite{OlshSafin}. For the analysis we need a tubular neighborhood $U_\delta$, with a radius $\delta$ that depends on $h$. We define $\delta_h:=c h$, with a constant $c>0$ that is fixed in the remainder. We assume that $h$ is sufficiently small such that $\Gamma_h \subset U_{\delta_h} \subset U_\delta$ holds, cf.~ \eqref{Dist}. Define  $\Phi_h:\,{\Omega}\to{\Omega}$ as (cf. Fig.~\ref{fig:Phi})
\begin{equation} \label{defPhi}
\Phi_h(\bx)=\left\{\begin{aligned} &
\bx-\bn(\bx)\frac{\delta_h^2-d(\bx)^2}{\delta_h^2-\gamma_h^e(\bx)^2}\gamma_h^e(\bx)&~~\text{if}~~ \bx\in \bar U_{\deltah},\\
&\bx&~~\text{if}~~ \bx\in \Omega\setminus\ U_{\deltah}.
\end{aligned}\right.
\end{equation}
We assume that $h$ is sufficiently small such that for all $\bx\in \bar U_{\deltah}$ the estimate $\deltah^2-\gamma_h^e(\bx)^2> \tilde c h^2 $ holds with  a mesh independent constant $\tilde c>0$.
Using this  and the definition in \eqref{defPhi}, we conclude that $\Phi_h$ is a bijection on $\Omega$ with the  properties:
\[
 \Phi_h(\Omega_{i,h})=\Omega_i,\quad \Phi_h(\bx)=\bp(\bx)~~\mbox{for}~\bx\in\Gamma_h,\quad  \bp(\Phi_h(\bx))=\bp(\bx)~~\mbox{for}~\bx\in U_{\deltah}.
\]
\begin{figure}[ht!]
  \begin{center}
  \begin{picture}(300,120)
    \put(0,-10){\includegraphics[width=0.7\textwidth,trim=0 0 0 100, clip]{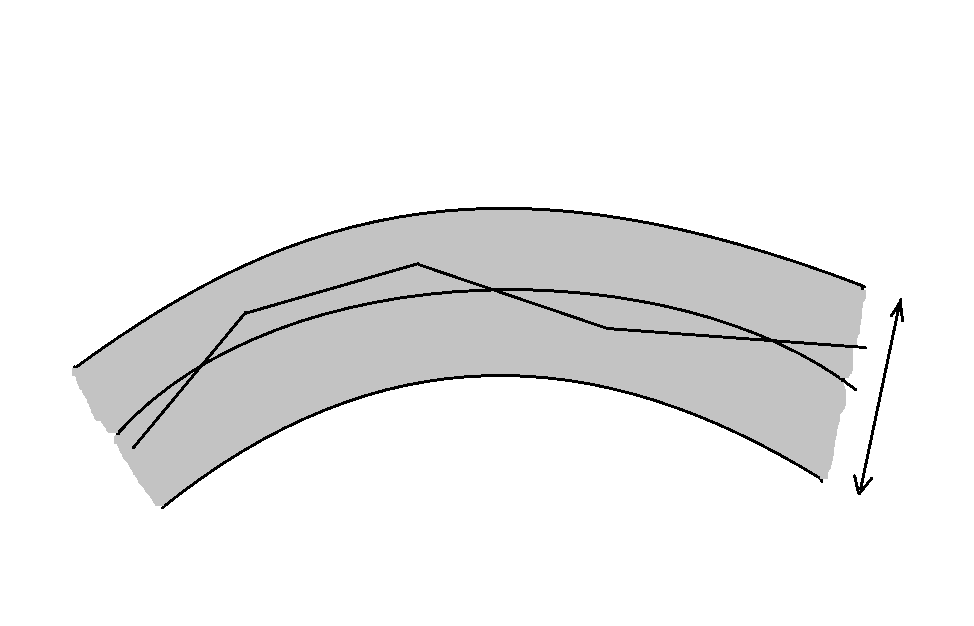}}
    \put(240,45){$O(h)$}
    \put(68,5){$\Gamma_h$}
    \put(67,7){\vector(-1,1){30}}
    \put(5,75){\vector(1,-1){30}}
    \put(7,80){$\Gamma$}
    \put(130,62){$U_{\deltah}$}
    \put(130,30){$\Phi_h=I$}
    \put(50,100){$\Phi_h=I$}
  \end{picture}
  \end{center}
  \caption{Illustration to the construction of $\Phi_h$: $\Phi_h$ gradually stretches $U_{\deltah}$ (grey region) along normal directions to $\Gamma$ such that $\Phi_h(\Gamma_h)=\Gamma$. }\label{fig:Phi}
\end{figure}
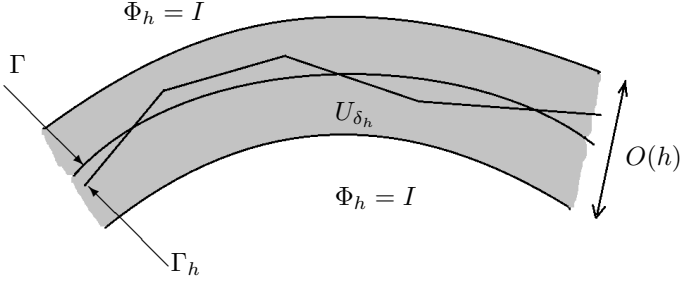

Some further properties of this mapping are derived in the following lemma.
\begin{lemma}\label{L_Phi} Consider $\Gamma_h$ as defined in \eqref{Omega_h}, with $\phi_h$ such that \eqref{phi_h} holds. The mapping $\Phi_h$ has the smoothness properties $\Phi_h\in \left(H^{1,\infty}(\Omega)\right)^3$,
$\Phi_h\in \left(H^{1,\infty}(\Gamma_h)\right)^3$.  Furthermore, for $h$ sufficiently small the estimates
\begin{align}\label{mapping}
\|\mbox{\rm id}-\Phi_h\|_{L^\infty(\Omega)}+ h \|\bI-\mathrm{D}\Phi_h\|_{L^\infty(\Omega)} & \le c\,h^{q+1}
\\
\|1-\mathrm{det}(\mathrm{D}\Phi_h)\|_{L^\infty(\Omega)} & \le c\,h^{q}\label{det}
\end{align}
hold, where $\mathrm{D}\Phi_h$ is the Jacobian matrix. For surface area elements we have
\begin{equation}\label{muh2}
d\bs(\Phi_h(\bx))=  \mu_h d\bs_h(\bx),\quad\bx\in\Gamma_h,\quad\text{with}~~\|1-\mu_h\|_{L^\infty(\Gamma_h)}\le c\,h^{q+1}.
\end{equation}
\end{lemma}
\begin{proof} The discrete surface $\Gamma_h$ is the graph of $\gamma_h(\bx)$, $\bx\in\Gamma$.
Therefore,   $\gamma_h(\bx)$ satisfies
\begin{equation}\label{implicit}
\phi_h(\bx+\gamma_h(\bx)\bn(\bx))=0,\quad \bx\in\Gamma.
\end{equation}
Since $\phi_h$ is continuous, piecewise polynomial and $\bn(\bx)\in C^q(U_\delta)^3$, $q\ge1$, the properties of an implicit
function imply that  $\gamma_h$ is continuous, piecewise $C^1$ and hence  $\gamma_h\in H^{1,\infty}(\Gamma)$.
The smoothness properties of $\Phi_h$ now follow from the construction.
To compute the surface gradient of $\gamma_h$,  we differentiate   this identity and using the chain rule we obtain the relation:
\[
\unabla \gamma_h(\bx) = -\frac{(I+\gamma_h(\bx)\bH(\bx))\unabla \phi_h(\bx')}{\bn(\bx)\cdot\nabla \phi_h(\bx')},\quad
\bx'=\bx+ \gamma_h(\bx)\bn(\bx),\quad \bx\in\Gamma.
\]
For the denominator in this expression we get,
using $|\bx - \bx'| \leq {\rm dist}(\Gamma_h,\Gamma) \leq c h^{q+1}$, \eqref{propphi}, \eqref{phi_h} and taking $h$ sufficiently small:
\begin{align*}
 |\bn(\bx)\cdot \nabla \phi_h(\bx')| &= \big|\bn(\bx)\cdot (\nabla \phi_h(\bx')-\nabla \phi(\bx'))+\bn(\bx)\cdot( \nabla \phi(\bx')-\nabla \phi(\bx))+ |\nabla \phi(\bx)|\big| \\
& \geq c_0 - ch^q \geq \frac12 c_0.
\end{align*}
For the nominator we use $\unabla \phi(\bx)=0$ and \eqref{phi_h} to get:
\begin{align*}
| \unabla \phi_h(\bx')| \leq  |\unabla (\phi_h(\bx')-\phi(\bx'))| +  |\unabla (\phi(\bx')- \phi(\bx))| \leq ch^q.
\end{align*}
From this and \eqref{Dist} we infer
\begin{equation}\label{eta_h}
\|\gamma_h\|_{L^\infty(\Gamma)}+ h \|\unabla\gamma_h\|_{L^\infty(\Gamma)}\le c\,h^{q+1}.
\end{equation}
 The following  surface
  area transformation property can be found in, e.g.,  \cite{Demlow09,DD07}:
\begin{align*}
  \mu_h(\bx) d \bs_h(\bx) & = d \bs(\bp(\bx)), \quad \bx \in \Gamma_h, \\ \mu_h(\bx)& := (1-d(\bx) \kappa_1(\bx))(1-d(\bx) \kappa_2(\bx)) \bn(\bx)^T \bn_h(\bx),
\end{align*}
with $\kappa_1,\kappa_2$ the nonzero eigenvalues of the Weingarten map and $\bn_h$ the unit normal on $\Gamma_h$. Note that
 $\Phi_h(\bx)=\bp(\bx)$ on $\Gamma_h$ holds. From \eqref{phi_h} we get $\|1- \mu_h\|_{L^\infty(\Gamma_h)} \leq ch^{q+1}$. Hence, the result in \eqref{muh2} holds.
For the term $l(\bx):=\frac{\delta_h^2-d(\bx)^2}{\delta_h^2-\gamma_h^e(\bx)^2}$, with $ \bx\in \bar U_{\deltah}$, used in \eqref{defPhi} we have
$\|l\|_{L^\infty(U_{\deltah})} \leq c$ and $\|\nabla l\|_{L^\infty(U_{\deltah})} \leq ch^{-1}$. Using these estimates and \eqref{nabla1}, \eqref{eta_h}
we obtain  $\|{\rm id}-\Phi_h\|_{L^\infty(\Omega)}\le ch^{q+1}$ and $\|\bI-\mathrm{D}\Phi_h\|_{L^\infty(\Omega)}\le ch^q$.
 This proves \eqref{mapping}.
The result in \eqref{det} immediately follows from \eqref{mapping}.
\quad
\end{proof}

\subsection{Smooth extensions}
For functions $v$ on $\Gamma$ we have introduced above the smooth constant extension along normals, denoted by $v^e$. Below we also need a smooth extension to $\Omega_{i,h}$ of functions $u$ defined on $\Omega_i$. This extension will also be denoted by $u^e$. Note that $u\circ \Phi_h$ defines an extension to $\Omega_{i,h}$. This extension, however, has smoothness $H^{1, \infty}(\Omega_{i,h})$, which is not sufficient for the interpolation estimates that we use further on. Hence, we introduce an extension $u^e$, which is close to $u\circ \Phi_h$ in the sense as specified in Lemma~\ref{lemmsm} and is more regular.

As mentioned above, we make the smoothness assumption $\Gamma\in C^{k+1}$, where $k$ is the degree of the polynomials used in the finite element space, cf.~\eqref{bulk}. We denote by $E_i$ a linear bounded extension operator $H^{k+1}(\Omega_i)\to H^{k+1}(\mathbb{R}^3)$ (see Theorem 5.4 in \cite{Wloka}). This operator satisfies
\begin{equation}\label{est_ext}
\|E_i u\|_{H^{m}(\mathbb{R}^3)}\leq c \|u\|_{H^{m}(\Omega_i)} \quad \forall~u \in H^{k+1}(\Omega_i), \quad m=0,\ldots, k+1, ~~i=1,2.
\end{equation}
For a piecewise smooth function $u\in H^{k+1}(\Omega_1\cup\Omega_2)$,  we denote by $u^e$ its
``transformation'' to a  piecewise smooth function $u^e\in H^{k+1}(\Omega_{1,h}\cup\Omega_{2,h})$ defined by
\begin{equation} \label{defexti}
u^e=\left\{\begin{array}{ll}
E_1(u|_{\Omega_1})&\quad \text{in}~~  \Omega_{1,h}\\
E_2(u|_{\Omega_2})&\quad \text{in}~~  \Omega_{2,h}.
\end{array}\right.
\end{equation}
The next lemma quantifies in which sense this  function $u^e$ is close to $u\circ \Phi_h$.
\begin{lemma} \label{lemmsm} The following estimates hold for $i=1,2$:
\begin{align}\label{differ_est}
\|u\circ \Phi_h-u^e\|_{\Omega_{i,h}} & \le c h^{q+1} \|u\|_{H^1(\Omega_i)},
 \\
\label{differ_est2}
\|(\nabla u)\circ \Phi_h-\nabla u^e\|_{\Omega_{i,h}} & \le c h^{q+1}  \|u\|_{H^2(\Omega_i)},
 \\
\label{differ_est3}
\|u\circ \Phi_h-u^e\|_{\Gamma_h} & \le c h^{q+1}  \|u\|_{H^2(\Omega_i)},
\end{align}
for all $u\in H^2(\Omega_i)$.
\end{lemma}
\begin{proof} Without loss of generality we consider $i=1$. Note that $u\circ \Phi_h =E_1(u|_{\Omega_1})\circ \Phi_h$ in $\Omega_{1,h}$ and
$u^e= E_1(u|_{\Omega_1})$ in $\Omega_{1,h}$. To simplify the notation, we write $u_1=E_1(u|_{\Omega_1})\in H^1(\mathbb{R}^3)$. We  use that $\Phi_h={\rm id}$ and $u=u^e$ on $\Omega\setminus\overline{U}_{\deltah}$ and transform to local coordinates in $U_{\deltah}$ using the co-area formula:
\begin{equation} \label{eq1} \begin{split}
\|u\circ \Phi_h-u^e\|_{\Omega_{1,h}}^2 & =\|u_1\circ \Phi_h-u_1\|_{\Omega_{1,h}}^2   =  \|u_1\circ \Phi_h-u_1\|_{\Omega_{1,h}\cap U_{\deltah}}^2
 \\ & =\int_{\Gamma}\int_{-\deltah}^{\gamma_h}(u_1\circ \Phi_h-u_1)^2 |\nabla \phi|^{-1} dr\,d\bs.
\end{split} \end{equation}
In local coordinates the mapping $\Phi_h$ can be represented as $\Phi_h(\bs,r)=(\bs,p_\bs(r))$, with
\[ p_\bs(r)=r - \frac{\deltah^2 - r^2}{\deltah^2 -\gamma_h(\bs)^2} \gamma_h(\bs).
\]
The function $p_\bs$ satisfies $|p_\bs(r)-r|\le ch^{q+1}$.
We use the identity
\begin{equation} \label{ident1}
(u_1\circ \Phi_h-u_1)(\bs,r)=\int_{r}^{p_\bs(r)}\br\cdot\nabla u_1(\bs,t)dt,\quad \br=\frac{\Phi_h(\bs,r)-(\bs,r)}{|\Phi_h(\bs,r)-(\bs,r)|}.
\end{equation}
Due to \eqref{eq1}, \eqref{ident1}, the Cauchy inequality and $|\nabla \phi|\geq c_0>0$ on $U_\delta$, we get
\begin{equation} \label{eq2}
\begin{split} \|u_1\circ \Phi_h-u_1\|_{\Omega_{1,h}}^2
& \le c \int_{\Gamma}\int_{-\deltah}^{\gamma_h}|p_\bs(r)-r| \int_{r}^{p_\bs(r)}|\nabla u_1(\bs,t)|^2\,|dt|\, dr\,d\bs\\
&\le ch^{q+1} \int_{\Gamma}\int_{-\deltah}^{\gamma_h} \int_{r-ch^{q+1}}^{r+ch^{q+1}}|\nabla u_1(\bs,t)|^2\,dt \, dr\,d\bs.
\end{split}
\end{equation}
Let
$\chi_{[-ch^{q+1},ch^{q+1}]}$ be the characteristic function on $[-ch^{q+1},ch^{q+1}]$ and  define $g(t)=|\nabla u_1(\bs,t)|^2$ for $t\in [-\deltah-ch^{q+1},\gamma_h+ch^{q+1}]$, $g(t)=0$, $t\notin [-\deltah-ch^{q+1},\gamma_h+ch^{q+1}]$. Applying the $L^1$-convolution inequality we get
\begin{align*}
& \int_{-\deltah}^{\gamma_h} \int_{r-ch^{q+1}}^{r+ch^{q+1}}|\nabla u_1(\bs,t)|^2\,dt dr \le c \int_{-\infty}^\infty \int_{-\infty}^\infty \chi_{[-ch^{q+1},ch^{q+1}]}(r-t) g(t) \, dt \, dr  \\  & \leq c \|\chi_{[-ch^{q+1},ch^{q+1}]}\|_{L^1(\mathbb{R})} \|g\|_{L^1(\mathbb{R})} \leq   ch^{q+1} \int_{-\deltah-ch^{q+1}}^{\gamma_h+ch^{q+1}} |\nabla u_1(\bs,t)|^2\, dt,
\end{align*}
and using this in \eqref{eq2} yields
\[
\begin{aligned}\|u_1\circ \Phi_h-u_1\|_{\Omega_{1,h}}^2
  &\le ch^{2q+2} \int_{\Gamma}\int_{-\deltah-ch^{q+1}}^{\gamma_h+ch^{q+1}} |\nabla E_1(u|_{\Omega_1})(\bs,t)|^2\, dr\,d\bs\\
& \le  c\,h^{2q+2} \|E_1(u|_{\Omega_1})\|_{H^1(\mathbb{R}^3)}^2
\le  c\,h^{2q+2} \|u\|_{H^1(\Omega_1)}^2.
\end{aligned}
\]
For deriving the estimate \eqref{differ_est2} we note that $(\nabla u)\circ \Phi_h=(E_i((\nabla u)|_{\Omega_i})\,)\circ \Phi_h = (\nabla(E_i( u|_{\Omega_i}))\,)\circ \Phi_h = (\nabla u_1) \circ \Phi_h$ in $\Omega_{i,h}$. Hence we have
$$\|(\nabla u)\circ \Phi_h-\nabla u^e\|_{\Omega_{1,h}}=\|(\nabla u_1)\circ \Phi_h-\nabla u_1\|_{\Omega_{1,h}}.$$
We can repeat the arguments used above, with $u_1$ replaced by $\frac{\partial u_1}{\partial x_j}$, $j=1,2,3$, and thus obtain the estimate \eqref{differ_est2}.

We continue to work in local coordinates and estimate the surface integral on the left-hand side of \eqref{differ_est3} using the Cauchy inequality and $\|\gamma_h\|_{L^\infty(\Gamma)}\le ch^{q+1}$:
\[
\begin{split}
\|u_1\circ \Phi_h- &u_1\|_{\Gamma_h}^2=\int_{\Gamma} (u_1- u_1\circ \Phi_h^{-1})^2(\bs,0)\mu_h^{-1}d\bs =
\int_{\Gamma} [u_1(\bs,0)- u_1(\bs,\gamma_h(\bs))]^2\mu_h^{-1}d\bs\\
 &= \int_{\Gamma}\left(\int_{0}^{\gamma_h}\br\cdot\nabla u_1(\bs,t)dt\right)^2\mu_h^{-1}d\bs
 \le c \int_{\Gamma}|\gamma_h| \int_{0}^{\gamma_h}|\nabla u_1(\bs,t)|^2\,|dt|\, d\bs\\
 &\le c h^{q+1}\int_{\Gamma} \int_{0}^{\gamma_h}|\nabla u_1(\bs,t)|^2\,|dt|\, d\bs\le c h^{q+1}\|u_1\|_{H^1(U_{h^{q+1}})}^2.
\end{split}
\]
Here $U_{h^{q+1}}$ is a tubular neighborhood of $\Gamma$ of width $O(h^{q+1})$ such that $\Gamma_h\subset U_{h^{q+1}}$.
Now we apply the following result, proven in Lemma 4.10 in \cite{ElliottRanner}:
\begin{equation*}
 \|w\|_{U_{h^{q+1}}}^2 \leq c h^{q+1} \|w\|_{H^1(\mathbb{R}^3)}^2 \quad \text{for all}~~w \in H^1(\mathbb{R}^3).
\end{equation*}
We let $w=\nabla u_1$ (componentwise) and use $\|E_1(u|_{\Omega_1})\|_{H^2(\mathbb{R}^3)}\le c \|u\|_{H^2(\Omega_1)}$ to prove \eqref{differ_est3}.
\end{proof}

\subsection{Approximate Galerkin orthogonality}
Due to the geometric errors, i.e., approximation of $\Omega_i$ by $\Omega_{i,h}$ and of $\Gamma$ by $\Gamma_h$, there is a so-called variational crime and only an \emph{approximate} Galerkin orthogonality relation holds. In this section we derive bounds for the deviation from orthogonality. The analysis is rather technical but the approach is similar to related analyses from the literature, e.g., \cite{DD07,Demlow09,BurmanHansbo}.

Let $(u,v) \in \bVt$ be the solution of the weak formulation \eqref{weak} and $(u_h,v_h)\in\bVt_h$ the discrete solution of \eqref{weakh}.
We take an arbitrary finite element test function $(\eta,\zeta)\in\bV_{h}$. We use $(\eta,\zeta)\circ\Phi_h^{-1}\in\bV$ as a test function in \eqref{weak} and then obtain the \emph{approximate Galerkin relation}:
\begin{align}
 & a_h((u^e-u_h,v^e-v_h);(\eta,\zeta))  \nonumber \\ & =  a_h((u^e,v^e);(\eta,\zeta))-a((u,v);(\eta,\zeta)\circ\Phi_h^{-1}) \label{term1} \\
 & \quad  + (f,\eta\circ\Phi_h^{-1})_{\Omega}+(g,\zeta\circ\Phi_h^{-1})_{\Gamma}-(f_h,\eta)_{\Omega}-(g_h,\zeta)_{\Gamma_h}. \label{term2}
\end{align}
In the analysis of the right-hand side of this relation we have to deal with full and tangential gradients $\nabla(\eta\circ\Phi_h^{-1})$, $\unabla(\zeta\circ\Phi_h^{-1})$.  For the full gradient in the bulk domains one finds
\begin{equation} \nabla(\eta\circ\Phi_h^{-1})(\bx) = \mathrm{D}\Phi_h(\by)^{-T}\nabla \eta(\by),\quad \bx \in \Omega, ~\by:=\Phi_h^{-1}(\bx). \label{tr1}
\end{equation}
 To handle the tangential gradient, a more subtle approach is required because one has to relate the tangential gradient $\unabla$ to $\unablah$. Let $\bn_h(\by)$, $\by \in \Gamma_h$, denote the unit normal on $\Gamma_h$ (defined a.e. on $\Gamma_h$). Furthermore, $\bP(\bx)=\bI - \bn(\bx) \otimes \bn(\bx)$ ($\bx \in U_\delta$), $\bP_h(\by)=\bI - \bn_h(\by) \otimes \bn_h(\by)$ ($\by \in \Gamma_h$).
Recall that $\unabla u(\bx)=\bP(\bx) \nabla u(\bx)$, $\unablah u(\by)=\bP_h(\by) \nabla u(\by)$.  We use the following relation, given in, e.g.,~\cite{DD07}: for $w \in H^1(\Gamma)$ it holds
\begin{equation} \label{Demrel}
\begin{split}
 \unabla w(\bp(\by)) & = \bB(\by) \unablah w^e(\by) \quad \text{a.e. on $\Gamma_h$}, \\
  \bB(\by) &= (\bI - d(\by) \bH(\by))^{-1}\tilde \bP_h(\by), \quad \tilde \bP_h(\by):=\bI- \frac{\bn_h(y)\otimes \bn(y)}{\bn_h(y)\cdot\bn(y)}.
\end{split}
\end{equation}
From the construction of the bijection $\Phi_h: \, \Gamma_h \to \Gamma$ it follows that $(\zeta \circ \Phi_h^{-1})^e(\by)=\zeta(\by)$ holds for  all $\by \in \Gamma_h$. Application of \eqref{Demrel} yields an interface analogon of the relation \eqref{tr1}:
\begin{equation} \label{tr2}
 \unabla (\zeta \circ \Phi_h^{-1})(\bx) = \bB(\by) \unablah \zeta (\by), \quad \bx \in \Gamma, \by= \Phi_h^{-1}(\bx ) \in \Gamma_h.
\end{equation}
The mapping $\Phi_h$ equals the identity outside  the (small) tubular neighborhood $U_{\deltah}$. In the analysis we want to make use of the fact that the width  behaves like $\deltah=ch$. For this, we again make use of the result in Lemma 4.10 in \cite{ElliottRanner}:
\begin{equation} \label{locEl}
 \|w\|_{U_{\deltah} \cap \Omega_{i}} \leq c h^\frac12 \|w\|_{H^1(\Omega_{i})} \quad \text{for all}~~w \in H^1(\Omega_{i}),
\end{equation}
where $c$ depends only on $\Omega_{i}$.
Using properties of $\Phi_h$, i.e. $\Phi_h(U_{\deltah}\cap \Omega_{i,h})=U_{\deltah}\cap \Omega_{i}$, and \eqref{tr1} we thus get, with $J_h=\det(D\Phi_h)$:
\begin{equation} \label{locEldiscr}
 \|w\|_{U_{\deltah} \cap \Omega_{i,h}} =  \|w\circ \Phi_h^{-1}J_h^{- \frac12}\|_{U_{\deltah}\cap \Omega_{i}} \leq c   h^\frac12  \|w\circ \Phi_h^{-1}\|_{H^1(\Omega_{i})} \leq c h^\frac12 \| w\|_{H^1(\Omega_{i,h})},
\end{equation}
for all $w \in H^1(\Omega_{i,h})$, with a constant $c$ independent of $w$ and $h$.

We introduce a convenient compact notation for the approximate Galerkin relation. We use $U:=(u,v)=(u_1,u_2,v)$, and similarly $U^e=(u^e,v^e)$, $U_h:=(u_h,v_h)\in \bV_h$, $\Theta=(\eta,\zeta) \in H^1(\Omega_{1,h}\cup \Omega_{2,h}) \times H^1(\Gamma_h)$.
Furthermore
\[ \begin{split}
 F_h(\Theta) & :=a_h(U^e;\Theta)-a(U;\Theta \circ \Phi_h^{-1}). 
\end{split}
\]
For the right-hand side in the discrete problem \eqref{weakh} we set
\begin{equation} \label{choicerhs} f_h:=J_h f\circ \Phi_h, \quad g_h:=\mu_h g\circ \Phi_h,
\end{equation}
with $J_h:=\det(D\Phi_h)$ and $\mu_h$ as in \eqref{muh2}.
We make this choice, because then the
consistency terms in \eqref{term2} vanish and
the approximate Galerkin relation takes the form
\begin{equation} \label{ApproxGal}
 a_h(U^e-U_h;\Theta_h)=F_h(\Theta_h) \quad \text{for all}~~\Theta_h \in \bV_h.
\end{equation}
The choice \eqref{choicerhs} requires explicit knowledge of the transformation $\Phi_h$, which may not be practical. We comment on alternatives in Remark~\ref{remalter}. \\
 In Lemma~\ref{lemA}   we derive a bound for the functional $F_h$. 
 \begin{lemma}\label{lemA}
For $m=0$ and $m=1$ the following estimate holds:
\[ \begin{split}
|F_h(\Theta)|=& \big| a_h(U^e;\Theta)-a(U;\Theta\circ\Phi_h^{-1}) \big|  \\ & \le c h^{q+m} \big(\|u\|_{H^2(\Omega_1 \cup \Omega_2)}+\|v\|_{H^1(\Gamma)} \big)\big(\|\eta \|_{H^{1+m}(\Omega_{1,h} \cup \Omega_{2,h})}+\|\zeta\|_{H^1(\Gamma_h)} \big)
 \end{split} \]
for all $U=(u,v) \in H^2(\Omega_1\cup\Omega_2) \times H^1(\Gamma)$,  $\Theta=(\eta,\zeta)\in H^{1+m}(\Omega_{1,h} \cup \Omega_{2,h}) \times H^1(\Gamma_h)$.
\end{lemma}
\begin{proof} Take $U=(u,v) \in H^2(\Omega_1\cup\Omega_2) \times H^1(\Gamma)$ and  $\Theta=(\eta,\zeta)\in H^{1+m}(\Omega_{1,h} \cup \Omega_{2,h}) \times H^1(\Gamma_h)$.
Denote $J_h=\det(D\Phi_h)$. Using the definitions of the bilinear forms, the relations \eqref{tr1}, \eqref{Demrel}, \eqref{tr2} and an integral transformation rule we get
\begin{align}
 & a_h(U^e;\Theta)-a(U;\Theta\circ\Phi_h^{-1})= a_h((u^e,v^e);(\eta,\zeta))-a((u,v);(\eta,\zeta)\circ\Phi_h^{-1}) \\ & =
 \sum_{i=1}^2 \Big[ (\nu_i\nabla u^e,\nabla \eta)_{\Omega_{i,h}} - (\nu_i\nabla u\circ\Phi_h,J_h (\mathrm{D}\Phi_h)^{-T} \nabla \eta)_{\Omega_{i,h}} \Big] \label{diff}
 \\  & + \sum_{i=1}^2  \frac12 \Big[(\bw_h \cdot \nabla u^e,\eta)_{\Omega_{i,h}}- ((\bw \cdot \nabla u)\circ\Phi_h,J_h\eta )_{\Omega_{i,h}} \label{conv1} \\
  &\qquad ~~ - (\bw_h\cdot\nabla\eta, u^e)_{\Omega_{i,h}} + \big((\bw\circ \Phi_h) \cdot(\mathrm{D}\Phi_h)^{-T} \nabla \eta ,J_h u\circ\Phi_h \big)_{\Omega_{i,h}}
 \Big] \label{conv2} \\
 & + \nu_\Gamma(\unablah v^e,\unablah \zeta )_{\Gamma_h} - \nu_\Gamma(\mu_h \bB^T \bB \unablah v^e, \unablah \zeta)_{\Gamma_h} \label{surfdiff} \\
 & +  \frac12 \Big[(\bw_h \cdot \unablah v^e,\zeta )_{\Gamma_h} - (\mu_h (\bw\circ \Phi_h) \cdot \bB \unablah v^e,\zeta )_{\Gamma_h}  \label{surfconv1} \\
 & \quad -(\bw_h \cdot \unablah\zeta, v^e)_{\Gamma_h} + (\mu_h(\bw\circ \Phi_h)\cdot \bB  \unablah \zeta , v^e )_{\Gamma_h}
 \Big] \label{surfconv2}  \\
 & + \sum_{i=1}^2\Big[(u_i^e-  q_iv^e,\eta_i-K\zeta)_{\Gamma_h}-
((u_i- q_iv)\circ\Phi_h,\mu_h(\eta_i-K\zeta))_{\Gamma_h}\Big]. \label{adsorption}
\end{align}
In this expression the different terms correspond to bulk diffusion, bulk convection, surface diffusion, surface convection and adsorption, respectively. We derive bounds for these terms. We start with the bulk diffusion term in \eqref{diff}. Using the estimates derived in the Lemmas~\ref{L_Phi},\,\ref{lemmsm} we get
\begin{align}
 & \sum_{i=1}^2 \Big| (\nu_i\nabla u^e,\nabla \eta)_{\Omega_{i,h}} - (\nu_i\nabla u\circ\Phi_h,J_h (\mathrm{D}\Phi_h)^{-T} \nabla \eta)_{\Omega_{i,h}} \Big| \nonumber \\
  & \leq \sum_{i=1}^2 \nu_i \big[|(\nabla (u^e -u\circ\Phi_h),J_h(\mathrm{D}\Phi_h)^{-T} \nabla \eta)_{\Omega_{i,h}}| + | (\nabla u^e,(\bI-J_h(\mathrm{D}\Phi_h)^{-T})\nabla \eta)_{\Omega_{i,h}}|\big] \nonumber \\
& \leq  c \sum_{i=1}^2 \big[ h^{q+1} \|u\|_{H^2(\Omega_i)} \|\nabla \eta\|_{\Omega_{i,h}}+ h^q \|u^e\|_{H^1(\Omega_{i,h})}\|\nabla \eta\|_{\Omega_{i,h}} \big]  \label{improve} \\
& \leq ch^q \|u\|_{H^2(\Omega_1\cup \Omega_2)} \|\eta\|_{H^1(\Omega_{1,h}\cup \Omega_{2,h})}.\label{improve2}
\end{align}
If $\eta \in H^2(\Omega_{1,h} \cup \Omega_{2,h})$ we can modify the estimate \eqref{improve} as follows. Note that $J_h(\mathrm{D}\Phi_h)^{-T}= \bI$ on $\Omega_{i,h} \setminus U_{\deltah}$. Hence, using \eqref{locEldiscr} we get
\begin{equation} \label{estimprove}
 \begin{split}
 & | (\nabla u^e,(\bI-J_h(\mathrm{D}\Phi_h)^{-T})\nabla \eta)_{\Omega_{i,h}}|= |(\nabla u^e,(\bI-J_h(\mathrm{D}\Phi_h)^{-T})\nabla \eta)_{U_{\deltah} \cap \Omega_{i,h}} | \\
 & \leq ch^q \|\nabla u^e\|_{U_{\deltah} \cap \Omega_{i,h}} \|\nabla \eta\|_{U_{\deltah} \cap \Omega_{i,h}} \leq c h^{q+1}\|u^e \|_{H^2(\Omega_{i,h})}\|\eta\|_{H^2(\Omega_{i,h})} \\
& \leq c h^{q+1}\|u \|_{H^2(\Omega_{i})}\|\eta\|_{H^2(\Omega_{i,h})}.
 \end{split}
\end{equation}
Thus, instead of \eqref{improve2} we then obtain the upper bound
\[
  ch^{q+1} \|u\|_{H^2(\Omega_1\cup \Omega_2)} \|\eta\|_{H^2(\Omega_{1,h}\cup \Omega_{2,h})}.
\]
For the bulk convection  term in \eqref{conv1}-\eqref{conv2} we get
\begin{align*}
  & \sum_{i=1}^2  \frac12 \big|(\bw_h \cdot \nabla u^e,\eta)_{\Omega_{i,h}}- ((\bw \cdot \nabla u)\circ\Phi_h,J_h\eta )_{\Omega_{i,h}}\big| \\
  &\quad  + \frac12 \big|(\bw_h\cdot\nabla\eta, u^e)_{\Omega_{i,h}} - \big((\bw\circ \Phi_h) \cdot (\mathrm{D}\Phi_h)^{-T} \nabla \eta ,J_h u\circ\Phi_h \big)_{\Omega_{i,h}}\big| \\
 & \leq \sum_{i=1}^2 \frac12 \big|(\bw_h  \cdot \nabla u^e-(\bw \cdot \nabla u)\circ\Phi_h,J_h\eta )_{\Omega_{i,h}}\big|+ \frac12 \big| (\bw_h \cdot \nabla u^e, (1-J_h)\eta )_{\Omega_{i,h}}\big| \\
 & + \frac12 \big| \big(\bw_h -(\bw\circ \Phi_h)\cdot(\mathrm{D}\Phi_h)^{-T} \nabla\eta, J_h u\circ\Phi_h)_{\Omega_{i,h}}\big| +\frac12 \big|(\bw_h\cdot\nabla\eta, u^e-J_hu\circ\Phi_h )_{\Omega_{i,h}}\big|
\end{align*}
The difference $\bw_h - \bw\circ \Phi_h$ can be bounded using the assumption that the original (unscaled) velocity $\bw$ is sufficiently smooth, $\bw \in H^{1, \infty}(\Omega)$. Using this, the relation \eqref{ident1} and the definition of $\bw_h$ we get $\|\bw_h - \bw\circ \Phi_h\|_{L^\infty(\Omega_{i,h})} \leq c h^{q+1}$.
 The first term on the right-hand side above  can be bounded using
\begin{align*}
 & \|\bw_h \cdot \nabla u^e-(\bw \cdot \nabla u)\circ\Phi_h\|_{\Omega_{i,h}}  = \|\bw_h \cdot \nabla u^e-(\bw \circ\Phi_h)\cdot (\nabla u\circ\Phi_h)\|_{\Omega_{i,h}} \\
 & \leq \|( \bw_h- \bw \circ \Phi_h)\cdot \nabla u^e\|_{\Omega_{i,h}} + \|(\bw \circ\Phi_h)\cdot (\nabla u^e- \nabla u\circ\Phi_h)\|_{\Omega_{i,h}} \\
 & \leq c h^{q+1} \|u\|_{H^2(\Omega_i)},
 \end{align*}
where in the last step we used results from the  Lemmas~\ref{L_Phi},\,\ref{lemmsm}. The other three terms can be estimated by using $\|1-J_h\|_{L^\infty(\Omega)} \leq ch^{q}$, $\|\bI-(\mathrm{D}\Phi_h)^{-T}\|_{L^\infty(\Omega)} \leq ch^q$ and $\|u^e-u\circ\Phi_h\|_{\Omega_{i,h}} \leq c h^{q+1} \|u\|_{H^1(\Omega_i)}$. Thus we get a bound
\begin{align*}
  & \sum_{i=1}^2  \frac12 \big|(\bw_h \cdot \nabla u^e,\eta)_{\Omega_{i,h}}- ((\bw \cdot \nabla u)\circ\Phi_h,J_h\eta )_{\Omega_{i,h}}\big| \\
  &  + \frac12 \big|(\bw_h\cdot\nabla\eta, u^e)_{\Omega_{i,h}} - \big((\bw\circ\Phi_h) \cdot (\mathrm{D}\Phi_h)^{-T} \nabla \eta ,J_h u\circ\Phi_h \big)_{\Omega_{i,h}}\big|
 \\ & \leq c h^q \|u\|_{H^2(\Omega_1 \cup \Omega_2)}\|\eta\|_{H^1(\Omega_{1,h}\cup \Omega_{2,h})}.
\end{align*}
If $\eta \in H^2(\Omega_{1,h} \cup \Omega_{2,h})$ we can apply an argument very similar to the one in \eqref{estimprove} and  obtain the following upper bound for the bulk convection term:
\[
 c h^{q+1} \|u\|_{H^2(\Omega_1 \cup \Omega_2)}\|\eta\|_{H^2(\Omega_{1,h}\cup \Omega_{2,h})}.
\]
For the surface diffusion term in \eqref{surfdiff} we introduce, for $\by \in \Gamma_h$, the matrix $\bA(\by):= \bP_h(\by)\big(\bI -\mu_h(\by) \bB(\by)^T\bB(\by)\big)\bP_h(\by)$. Using $|d(\by)| \leq ch^{q+1}$, $|1- \mu_h(\by)| \leq ch^{q+1}$ and $\bP_h \tilde \bP_h=\bP_h$ we get $\|\bA\|_{L^\infty(\Gamma_h)} \leq ch^{q+1}$ and thus:
\begin{align*}  &\nu_\Gamma \big|(\unablah v^e,\unablah \zeta )_{\Gamma_h} - (\mu_h \bB^T \bB \unablah v^e, \unablah \zeta)_{\Gamma_h}\big|
 = \nu_\Gamma \big|(\bA\unablah v^e,\unablah \zeta )_{\Gamma_h}\big|\\ & \leq \|\bA\|_{L^\infty(\Gamma_h)} \|\unablah v^e\|_{\Gamma_h} \|\unablah\zeta\|_{\Gamma_h}
 \leq c h^{q+1} \|v\|_{H^1(\Gamma)} \|\zeta\|_{H^1(\Gamma_h)}.
\end{align*}
For the derivation of a bound for the surface convection term in \eqref{surfconv1}-\eqref{surfconv2} we introduce $\tilde \bw:=\bP_h(\bw_h - \mu_h\bB^T (\bw \circ \Phi_h))$. Using the results in Lemmas~\ref{L_Phi},\,\ref{lemmsm} and $\bP \bw= \bw$ it follows that
\[ \begin{split}  & \|\tilde \bw\|_{L^\infty(\Gamma_h)}  \leq \|\bP_h(\bw_h - \bw \circ \Phi_h)\|_{L^\infty(\Gamma_h)} +\|\bP_h(\bI-\mu_h\bB^T )(\bw \circ \Phi_h)\|_{L^\infty(\Gamma_h)} \\ & \leq   c\|\bP_h(\bI - \tilde \bP^T)\bP\|_{L^\infty(\Gamma_h)} + ch^{q+1} \leq
 c \|\bP_h \bn\|_{L^\infty(\Gamma_h)}  \|\bP \bn_h\|_{L^\infty(\Gamma_h)}   + ch^{q+1}  \\
 & = c\|(\bP_h-\bP) \bn\|_{L^\infty(\Gamma_h)}  \|(\bP- \bP_h) \bn_h\|_{L^\infty(\Gamma_h)}   + ch^{q+1}  \\ & \leq c\|\bP_h-\bP\|_{L^\infty(\Gamma_h)}^2+ ch^{q+1}.
   \end{split}
\]
 Using
\[
  |\bn_h(\by)-\bn(\by)| = |\bn_h(\by)-\bn(\bp(\by))|= \left| \frac{\nabla \phi_h(\by)}{|\nabla \phi_h(\by)|} -\frac{\nabla \phi(\bp(\by))}{|\nabla \phi(\bp(\by))|}\right|
\]
in combination with $|\by- \bp(\by)| \leq ch^{q+1}$ and the approximation error bound \eqref{phi_h} we get $\|\bP - \bP_h\|_{L^\infty(\Gamma_h)} \leq ch^q$. Hence,
 for the surface convection term in \eqref{surfconv1} we obtain
\begin{align*}
 & \big| (\bw_h  \cdot \unablah v^e,\zeta )_{\Gamma_h} - (\mu_h (\bw\circ \Phi_h) \cdot \bB \unablah v^e,\zeta )_{\Gamma_h} \big | \\
& = |(\tilde \bw\cdot \unablah v^e,\zeta )_{\Gamma_h}\big| \leq \|\tilde \bw\|_{L^\infty(\Gamma_h)}\|\unablah v^e\|_{\Gamma_h} \|\zeta\|_{\Gamma_h} \leq c h^{q+1} \|v\|_{H^1(\Gamma)}   \|\zeta\|_{H^1(\Gamma_h)}.
\end{align*}
The term in \eqref{surfconv2} can be bounded in the same way. Finally we consider the adsorption-desorption term in \eqref{adsorption}. Using the results in Lemmas~\ref{L_Phi},\,\ref{lemmsm} we get
\begin{align*}
& \sum_{i=1}^2\Big|(u_i^e-  q_iv^e,\eta_i-K\zeta)_{\Gamma_h}-
((u_i- q_iv)\circ\Phi_h,\mu_h(\eta_i-K\zeta))_{\Gamma_h}\Big| \\
 & \leq  \sum_{i=1}^2\big| (u_i^e - \mu_h u_i \circ\Phi_h,\eta_i-K\zeta)_{\Gamma_h}\big| + q_i \big|((1- \mu_h)v^e,\eta_i-K\zeta)_{\Gamma_h}\big| \\
& \leq c h^{q+1}  \sum_{i=1}^2\big( \|u_i\|_{H^2(\Omega_i)} + \|v\|_\Gamma\big)\big( \|\eta_i\|_{\Gamma_h} + \|\zeta\|_{\Gamma_h}\big).\\
& \leq c\,h^{q+1}(\|u\|_{H^2(\Omega_1\cup \Omega_2)}+ \|v\|_\Gamma)(\|\eta\|_{H^1(\Omega_{1,h}\cup \Omega_{2,h})}+\|\zeta\|_{\Gamma_h}).
\end{align*}
Combining these estimates for the terms in \eqref{diff}-\eqref{adsorption} completes the proof.
\quad\end{proof}
\medskip

From the arguments in the proof above one easily sees that the bounds derived in Lemma~\ref{lemA} also hold if in $a_h(\cdot;\cdot)$ and $a(\cdot;\cdot)$ the arguments are interchanged. This proves the result in the following lemma, that we need in the $L^2$-error analysis.
\begin{lemma}\label{lemA1}
The following estimate holds:
\[ \begin{split}
 & \big| a_h(\Theta;U^e)-a(\Theta\circ\Phi_h^{-1};U) \big|  \\ & \le c h^{q} \big(\|\eta \|_{H^{1}(\Omega_{1,h} \cup \Omega_{2,h})}+\|\zeta\|_{H^1(\Gamma_h)} \big) \big(\|u\|_{H^2(\Omega_1 \cup \Omega_2)}+\|v\|_{H^1(\Gamma)} \big)
 \end{split} \]
for all $U=(u,v) \in H^2(\Omega_1\cup\Omega_2) \times H^1(\Gamma)$,  $\Theta=(\eta,\zeta)\in H^{1}(\Omega_{1,h} \cup \Omega_{2,h}) \times H^1(\Gamma_h)$.
\end{lemma}

\medskip

As an immediate corollary of Lemma~\ref{lemA}, the definition of $F_h$ and the regularity estimate in Theorem~\ref{Th_wp} we obtain the following result.
 \begin{lemma}\label{lemcombine}
Assume that the solution $(u,v)$ of \eqref{weak} has smoothness $u\in H^2(\Omega_1\cup\Omega_2)$, $v\in H^2(\Gamma)$.
For $m=0$ and $m=1$ the following holds:
\begin{equation}\label{consist_est}
|F_h(\Theta)|\le c h^{q+m} (\|f\|_\Omega+\|g\|_\Gamma\big) \big(\|\eta \|_{H^{1+m}(\Omega_{1,h} \cup \Omega_{2,h})}+\|\zeta\|_{H^1(\Gamma_h)} \big)
 \end{equation}
for all   $\Theta=(\eta,\zeta)\in H^{1+m}(\Omega_{1,h} \cup \Omega_{2,h}) \times H^1(\Gamma_h)$.
\end{lemma}
\smallskip

\begin{remark} \label{remalter} \rm The consistency estimate in Lemma~\ref{lemcombine} is proved for the particular choice of the finite element
problem right-hand side as in \eqref{choicerhs}. This choice simplifies the analysis because the terms in \eqref{term2} vanish. Similar estimates, however,  can be proved if $f_h$ and $g_h$
are chosen as generic smooth extensions of $f\in H^1(\Omega_1\cup\Omega_2)$ and $g \in L^2(\Gamma)$. For example, one may
set $f_h=f^e$, $g_h=g^e|_{\Gamma_h}-c_f$, where $c_f$ is such that mean value condition is satisfied.
In this case,   the following estimate holds for the terms in \eqref{term2}:
\begin{equation*} \begin{split}
 & \big|(f,\eta\circ\Phi_h^{-1})_{\Omega}+(g,\zeta\circ\Phi_h^{-1})_{\Gamma}-(f^e,\eta)_{\Omega}-(g^e,\zeta)_{\Gamma_h}+(c_f,\zeta)_{\Gamma_h}\big| \\ & \leq  ch^{q+1} \big(\|f\|_{H^1(\Omega_1 \cup \Omega_2)} +\|g\|_\Gamma\big) \big( \|\eta\|_{H^1(\Omega_{1,h}\cup \Omega_{2,h})} + \|\zeta\|_{\Gamma_h}\big)
\end{split} \end{equation*}
for all $(\eta,\zeta)\in H^{1}(\Omega_{1,h} \cup \Omega_{2,h}) \times H^1(\Gamma_h)$. A proof of this estimate is given in   the preprint version of this paper
\cite{ThisPreprint}. The discretization error bounds in  \eqref{err_estH1} and \eqref{errL2} hold for this alternative right-hand side choice, provided  $\|f\|_\Omega$ is replaced by $\|f\|_{H^1(\Omega_1 \cup \Omega_2)}$.
\end{remark}

\subsection{Discretization error bound in the $\bV_h$-norm}

Based on the stability, continuity and approximate Galerkin properties presented in the previous sections, we derive a discretization  error bound in the $\bV_h$ norm. Due to the approximation of the interface the discrete solution $U_h=(u_h,v_h)$ has a domain that differs from that of the solution $U=(u,v)$ to the continuous problem. Therefore, it is not appropriate to define the error as $U-U_h$. It is natural to define the discretization error either as $U^e-U_h$, with functions defined on the domain corresponding to the discrete problem, or as $U-U_h \circ \Phi_h^{-1}$, with functions defined on the domain corresponding to the continuous problem. We use the former definition.
 In the analysis we need suitable interpolation operators, applicable to $U^e$.
The function $u^e$ consists of the pair $u^e=(E_1(u_{|\Omega_1})_{|\Omega_{1,h}},E_2(u_{|\Omega_2})_{|\Omega_{2,h}})=:(u^e_1,u^e_2)$, cf. \eqref{defexti}.
 As is standard in analyses of XFEM (or unfitted FEM) we define an interpolation based on the standard nodal interpolation of the smooth extension in the bulk space  $ V_h^{\rm bulk}$. Let $I_h^{\rm bulk}$ denote the nodal interpolation in  $ V_h^{\rm bulk}$ (which consists of  finite elements of degree $k$). We define $I_h u^e \in V_{\Omega,h}$ as follows:
\[
  I_hu^e= \big( [I_h^{\rm bulk}E_1(u_{|\Omega_1})]_{|\Omega_{1,h}},[I_h^{\rm bulk}E_2(u_{|\Omega_2})]_{|\Omega_{2,h}}\big).
\]
The construction of this operator and interpolation error bounds for $I_h^{\rm bulk}$ immediately yield
\begin{equation} \label{intu}
 \|I_hu^e-u^e\|_{H^1(\Omega_{1,h} \cup \Omega_{2,h})} \le c\, h^{k}  \|u\|_{H^{k+1}(\Omega_1 \cup \Omega_2)}.
\end{equation}
For the interpolation of $v^e$ we use a similar approach, namely $I_h v^e:=[I_h^{\rm bulk}v^e]_{|\Gamma_h}$. Here the interpolation operator $I_h^{\rm bulk}$ is applied only on the tetrahedra that are intersected by $\Gamma_h$. Interpolation error bounds for this operator are known in the literature, see, e.g.,  Theorem 4.2 in \cite{Reusken2014}:
\begin{equation}\label{intv}
 \|I_h v^e- v^e\|_{H^1(\Gamma_h)} \le c\, h^{k}  \|v\|_{H^{k+1}(\Gamma)}.
\end{equation}
Using these interpolation error bounds we obtain the following main theorem.
\begin{theorem}\label{Th_Conv1}  Let the solution $(u,v)\in\bVt$ of \eqref{weak2} be sufficiently smooth. For the finite element solution $(u_h,v_h)\in\bVt_h$ the following error estimate holds:
\begin{equation}\label{err_estH1}
\begin{split}
\|(u^e-u_h,v^e-v_h)\|_{\bV_h} & \le c h^k \big(\|u\|_{H^{k+1}(\Omega)} + \|v\|_{H^{k+1}(\Gamma)}\big)  \\
   &\,  + c h^q \big(\|f\|_{\Omega} + \|g\|_{\Gamma}\big),
\end{split} \end{equation}
where  $k$ is the degree of the finite element polynomials and $q$ the geometry approximation order  defined in \eqref{phi_h}.
\end{theorem}
\begin{proof} We use arguments similar to the second Strang's lemma.  Recalling the stability and continuity results from \eqref{infsup2h}, \eqref{conth}, the consistency error bound in Lemma~\ref{lemcombine} and the interpolation error bounds in \eqref{intu}, \eqref{intv} we get, with $I_hU^e=(I_hu^e, I_hv^e)$:
\begin{equation}\label{Th1_aux1}
\begin{split}
&\|I_hU^e-U_h\|_{\bV_h}\le C_{st}^{-1} \sup_{\Theta_h \in\bV_{h}}\frac{a_h(I_hU^e-U_h;\Theta_h)}{\|\Theta_h\|_{\bV_h}}\\
&=
C_{st}^{-1} \sup_{\Theta_h\in\bV_{h}}\Big(\frac{a_h(I_hU^e -U^e;\Theta_h)}{\|\Theta_h\|_{\bV_h}} +
\frac{ a_h(U^e-U_h;\Theta_h)}{\|\Theta_h\|_{\bV_h}}\Big)\\
&\le c\,\Big(\|I_hU^e-U^e\|_{\bV_h} +\sup_{\Theta_h \in\bV_{h}}\frac{
F_h(\Theta_h)}{\|\Theta_h\|_{\bV_h}}
\Big) \\
& \leq c\,h^k (\|u\|_{H^{k+1}(\Omega)}+\|v\|_{H^{k+1}(\Gamma)}) +  c h^q (\|f\|_\Omega+\|g\|_\Gamma).
\end{split}
\end{equation}
The desired result now follows by a triangle inequality and applying the interpolation error estimates \eqref{intu}, \eqref{intv} once more.
\end{proof}

\section{Error estimate in $L^2$-norm}\label{S_errorL2}
In this section we use a duality argument to show higher order convergence of the unfitted finite element method in the $\mathbf{L}^2$ product norm. As typical in the analysis of elliptic PDEs with Neumann boundary conditions, one considers the $L^2$ norm in a factor space:
\[
\|U\|_{\mathbf{L}^2/\mathbb{R}}=\inf_{\gamma\in\mathbb{R}}\|U-\gamma(q_1,q_2,1)\|_{L^2(\Omega) \times L^2(\Gamma)},\quad\text{for}~U\in L^2(\Omega) \times L^2(\Gamma),
\]
and $q_1,q_2\in[0,1]$ from \eqref{Totaltrans}.
A similar norm can be defined on $L^2(\Omega) \times L^2(\Gamma_h)$.

Define the error $E_h:=(U^e-U_h)\circ \Phi_h^{-1}\in L^2(\Omega) \times L^2(\Gamma)$. There is a constant
$\gamma\in\mathbb{R}$ such that $\tilde E_h:=E_h-\gamma(q_1,q_2,1)$ satisfies the consistency condition \eqref{rhs_gaude_Adjoint}.
According to Theorem~\ref{Th_wpd}  the dual problem: Find $W \in \bV_\alpha $ such that
\begin{equation} \label{dual}
  a(\Theta;W)=(\tilde E_h,\Theta)_{L^2(\Omega) \times L^2(\Gamma)} \quad \text{for all}~\Theta \in \bV,
\end{equation}
has the unique solution $W=(w,z)\in H^2(\Omega_1\cup \Omega_2)  \times H^2(\Gamma)$,
satisfying
\begin{equation} \label{regdual}
 \|w\|_{H^{2}(\Omega_1\cup\Omega_2)}+\|z\|_{H^{2}(\Gamma)} \leq c \|\tilde E_h\|_{L^2(\Omega) \times L^2(\Gamma)},
\end{equation}
with a constant $c$ independent of $\tilde E_h$.
\begin{theorem}\label{ThL2}  Let the assumptions in Theorem~\ref{Th_Conv1} and Theorem~\ref{Th_wpd} be fulfilled. For the finite element solution $(u_h,v_h)\in\bVt_h$ the following error estimate holds:
\begin{equation}\label{errL2}
\begin{split}
\|u^e-u_h,v^e-v_h\|_{\mathbf{L}^2/\mathbb{R}} & \le c h^{k+1} \big(\|u\|_{H^{k+1}(\Omega_1 \cup \Omega_2)} + \|v\|_{H^{k+1}(\Gamma)}\big)  \\
   &\,  + c h^{q+1} \big(\|f\|_\Omega + \|g\|_{\Gamma}\big),
\end{split}
\end{equation}
where  $k$ is the degree of the finite element polynomials and $q$ the geometry approximation order  defined in \eqref{phi_h}.
\end{theorem}
\begin{proof} First, let $\gamma_{\rm opt}:=\arg\inf_{\gamma\in\mathbb{R}}\|E_h-\gamma(q_1,q_2,1)\|_{L^2(\Omega) \times L^2(\Gamma)}$. Observe
the chain of estimates:
\[
\begin{aligned}
\|U^e-U_h\|_{\mathbf{L}^2/\mathbb{R}}&\le \|U^e-U_h-\gamma_{\rm opt}(q_1,q_2,1) \|_{L^2(\Omega) \times L^2(\Gamma_h)}\\
&=\|(U^e-U_h-\gamma_{\rm opt}(q_1,q_2,1))\circ \Phi_h^{-1}J_h^{-\frac12} \|_{L^2(\Omega) \times L^2(\Gamma)}
\\ &\le c\|(U^e-U_h-\gamma_{\rm opt}(q_1,q_2,1))\circ \Phi_h^{-1}\|_{L^2(\Omega) \times L^2(\Gamma)}\\
&= c\|(U^e-U_h)\circ \Phi_h^{-1}-\gamma_{\rm opt}(q_1,q_2,1)\|_{L^2(\Omega) \times L^2(\Gamma)}\\
&= c\,\|E_h\|_{\mathbf{L}^2/\mathbb{R}}  \leq c \|\tilde E_h\|_{L^2(\Omega)\times L^2(\Gamma)}.
\end{aligned}
\]
We apply the standard duality argument and thus obtain:
\begin{align*}
 &  \|\tilde E_h\|_{L^2(\Omega) \times L^2(\Gamma)}^2 = a(\tilde E_h,W) = a(E_h,W) \\
  &= a(E_h;W)-a_h(U^e-U_h;W^e) +a_h(U^e-U_h;W^e -I_hW^e) -a_h(U^e-U_h;I_hW^e) \\
  &= \big[a(E_h;W)-a_h(U^e-U_h;W^e)\big] +a_h(U^e-U_h;W^e -I_hW^e) + F_h(I_hW^e) \\
 & = \big[a(E_h;W)-a_h(U^e-U_h;W^e)\big] +a_h(U^e-U_h;W^e -I_hW^e) \\ & \quad  +F_h(I_hW^e-W^e)+F_h(W^e).
\end{align*}
These terms can be estimated as follows. For the term between square brackets we use Lemma~\ref{lemA1} and \eqref{regdual}:
\[ \begin{split}
 |a(E_h;W)-a_h(U^e-U_h;W^e)| & \leq ch^q \|U^e-U_h\|_{\bV_h}(\|w\|_{H^{2}(\Omega_1\cup\Omega_2)}+\|z\|_{H^{1}(\Gamma)})  \\ & \leq c h^q \|U^e-U_h\|_{\bV_h}\|\tilde E_h\|_{L^2(\Omega) \times L^2(\Gamma)}
\end{split} \]
For the second term we use continuity, the interpolation error  bound and \eqref{regdual}:
\[ \begin{split}
 |a_h(U^e-U_h;W^e -I_hW^e)| & \leq c h \|U^e-U_h\|_{\bV_h}(\|w\|_{H^{2}(\Omega_1\cup\Omega_2)}+\|z\|_{H^{2}(\Gamma)})  \\ & \leq c h \|U^e-U_h\|_{\bV_h}\|\tilde E_h\|_{L^2(\Omega) \times L^2(\Gamma)}.
\end{split} \]
For the third term we use Lemma~\ref{lemcombine} with $m=0$, the interpolation error  bound and \eqref{regdual}:
\[ \begin{split}
 |F_h(I_hW^e-W^e)| & \leq ch^q \big(\|f\|_\Omega + \|g\|_{\Gamma}\big)\|I_hW^e-W^e\|_{\bV_h} \\
  & \leq ch^{q+1} \big(\|f\|_\Omega+ \|g\|_{\Gamma}\big)\|\tilde E_h\|_{L^2(\Omega) \times L^2(\Gamma)}.
\end{split} \]
For the fourth term we use  Lemma~\ref{lemcombine} with $m=1$ and \eqref{regdual}:
\[
 |F_h(W^e)| \leq ch^{q+1}\big(\|f\|_\Omega + \|g\|_{\Gamma}\big)\|\tilde E_h\|_{L^2(\Omega) \times L^2(\Gamma)}.
\]
Combining these results and using the bound for $\|U^e-U_h\|_{\bV_h}$ of Theorem~\ref{Th_Conv1} completes the proof.
\end{proof}

\section{Numerical results}
We consider the stationary coupled bulk-interface convection diffusion problem \eqref{Total}  in the domain $\Omega=[-1.5, 1.5]^3$ and with the unit sphere $\Gamma=\{x\in\Omega:~ \|x\|_2=1\}$ as interface.
For the velocity field we take a rotating field in the $x$-$z$ plane: $\bw=\frac{1}{10}(z,0,-x)$. This $\bw$ satisfies the conditions \eqref{div_w} and \eqref{w.n}, i.e., $\Div\bw=0$ in $\Omega$ and $\bw \cdot \bn=0$ on $\Gamma$. On some parts of the boundary $\partial\Omega$ the velocity field $\bw$ is pointing inwards the domain, i.e., \eqref{cond2} does not hold. Hence, there are convective fluxes on $\partial \Omega$ and thus a Neumann boundary condition as in \eqref{Total} is not natural. For this reason, and to simplify the implementation, we use Dirichlet boundary conditions on $\partial\Omega$. Note that in this case we do not need the additional condition \eqref{Gau} to obtain well-posedness. For the scaling constant we take $K=1$.

\subsection{Convergence study}
In this experiment, the material parameters are chosen as $\nu_1=0.5$, $\nu_2=1$, $\nu_\Gamma=1$ and $\tilde k_{1,a}=0.5$, $\tilde k_{2,a}=2$, $\tilde k_{1,d}=2$, $\tilde k_{2,d}=1$.
The source terms $f_i\in L^2(\Omega)$, $i=1,2$, and $g\in L^2(\Gamma)$  in \eqref{Total} and the Dirichlet boundary data are taken such  that the exact solution of the coupled system is given by
\begin{equation} \label{exsol}
  \begin{split}
  v(x,y,z) &= 3x^2y - y^3, \\
  u_1(x,y,z) &= 2 u_2(x,y,z), \\
  u_2(x,y,z) &= e^{1-x^2-y^2-z^2} v(x,y,z).
\end{split}
\end{equation}
Note that the gauge condition \eqref{consfg} is satisfied for this choice of $f$ and $g$.
For the initial triangulation, $\Omega$ is divided into $4\times4\times4$ sub-cubes each consisting of 6 tetrahedra.
This initial mesh is uniformly refined up to 4 times, yielding $\T_h$. The discrete interface $\Gamma_h$ is obtained by linear interpolation of the signed distance function corresponding to $\Gamma$.  We use the finite element spaces in \eqref{bulk}-\eqref{FEspace} with $k=1$, i.e., $V_h^\mathrm{bulk}$ consists of piecewise linears on $\T_h$.\\
For the representation of functions in the trace finite element spaces $V_{\Gamma,h}$ and $V_{i,h}$, cf.~\eqref{FEspace}, we use the standard nodal basis functions in the bulk space $V_h^\mathrm{bulk}$. For $V_{\Gamma,h}$ and $V_{i,h}$ only those basis functions are used, the support of which has a nonzero intersection with $\Gamma_h$ and $\Omega_{1,h}$, respectively. Hence, the nodal finite element functions used for represented functions from the trace space $V_{\Gamma,h}$ are a subset of the basis functions that are used for representing functions from $V_{i,h}$.
The resulting coupled linear system is iteratively solved by a Generalized Conjugate Residual (GCR) method using a block diagonal preconditioner, where the bulk and interface systems are preconditioned by the symmetric Gauss-Seidel method.

The numerical solution $u_h, v_h$ after 2 grid refinements and the resulting interface approximation $\Gamma_h$ are shown in Figure~\ref{fig:numSol}.
\begin{figure}[ht!]
  \begin{center}
    \includegraphics[width=0.8\textwidth]{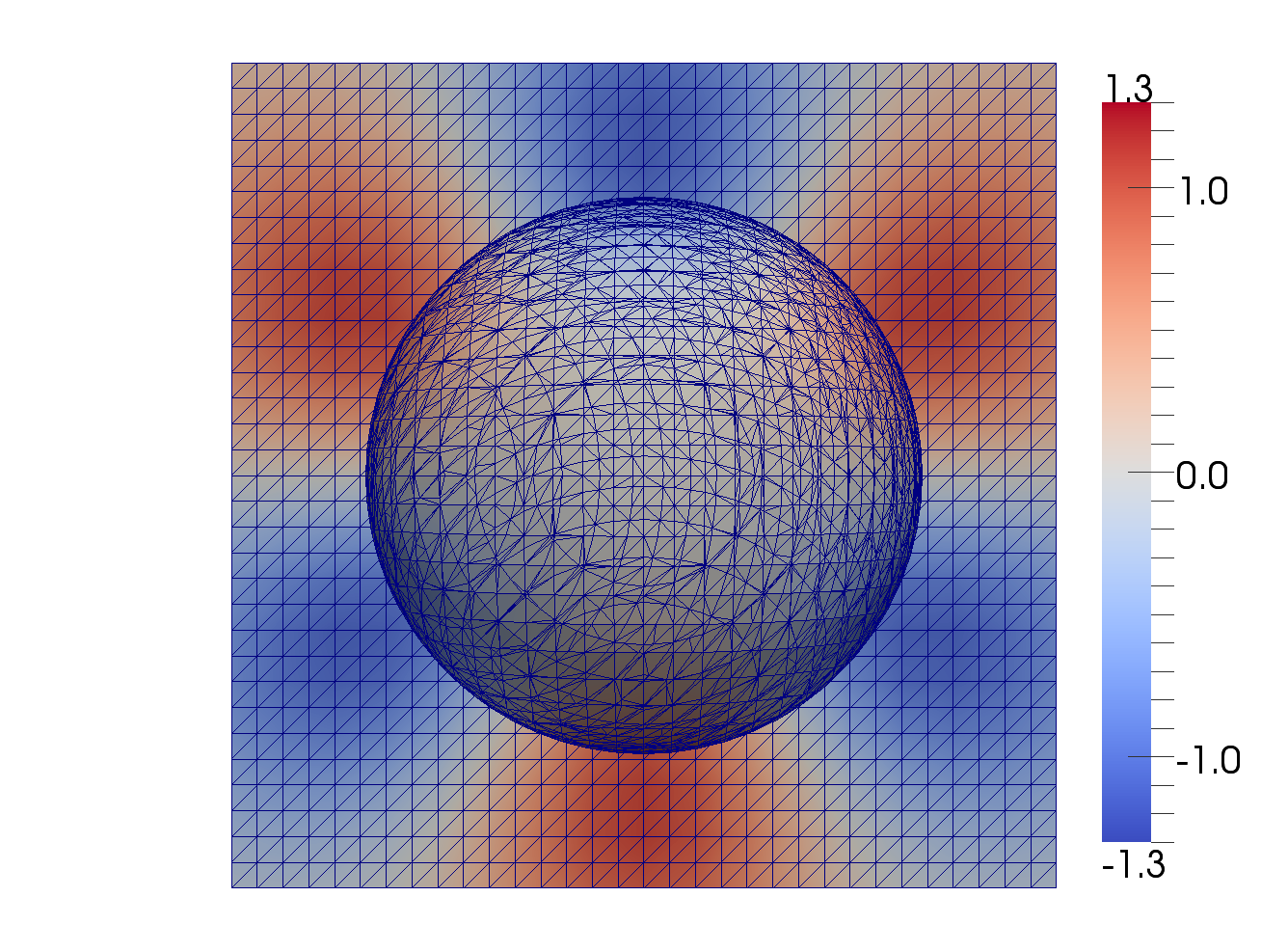}
  \end{center}
  \caption{Numerical solutions $v_h$ on $\Gamma_h$ and $u_h$ visualized on a cut plane $z=0$ for refinement level 2.}\label{fig:numSol}
\end{figure}

The $L^2$ and $H^1$ errors for the bulk and interface concentration are given in Tables~\ref{tab:bulkError} and \ref{tab:ifaceError}. For computing the discretization error on the ``discrete'' domains $\Omega_{i,h}$ and $\Gamma_h$ we use $u_{|\Omega_{i,h}}:=u_i$, $i=1,2$, $v_{|\Gamma_h}:=v$, with $u_i$ and $v$ as in  \eqref{exsol}.  In the column ``order'' we give the estimated $p$ in the error behavior ansatz $ch^p$. As expected, first order convergence is obtained for the $H^1$ errors of bulk and interface concentration, cf. Theorem~\ref{Th_Conv1}. The respective $L^2$ errors are of second order, which confirms the theoretical findings in Theorem~\ref{ThL2}.
\begin{table}[ht!]
\begin{center}
  \begin{tabular}{ccccc}
    \hline
    \# ref. & $\| u-u_h \|_{L^2(\Omega)}$ & order & $\| u-u_h \|_{H^1(\Omega_{1,h} \cup \Omega_{2,h})}$ & order
    \\\hline\\
    0	& 1.27E+0	& ---	& 6.07E+0	& ---  \\
    1	& 4.39E-1	& 1.53	& 3.39E+0	& 0.84 \\
    2	& 1.28E-1	& 1.78	& 1.79E+0	& 0.92 \\
    3	& 3.34E-2	& 1.94	& 9.12E-1	& 0.98 \\
    4	& 8.47E-3	& 1.98	& 4.58E-1	& 0.99
    \\\hline\\
  \end{tabular}
  \caption{$L^2$ and $H^1$ errors for bulk concentration $u_h$ on different refinement levels.\label{tab:bulkError}}
\end{center}
\end{table}
\begin{table}[ht! ]
\begin{center}
  \begin{tabular}{cccccc}
  \hline
    \# ref. & \# iter. & $\|v-v_h\|_{L^2(\Gamma_h)}$ & order & $\|v-v_h\|_{H^1(\Gamma_h)}$ & order
    \\\hline\\
    0	& 29	& 6.75E-1	& ---	& 4.60E+0	& ---  \\
    1	& 41	& 1.88E-1	& 1.85	& 2.30E+0	& 1.00 \\
    2	& 67	& 5.39E-2	& 1.80	& 1.01E+0	& 1.18 \\
    3	& 119	& 1.34E-2	& 2.01	& 5.15E-1	& 0.98 \\
    4	& 218	& 3.38E-3	& 1.99	& 2.52E-1	& 1.03 \\
    \hline\\
  \end{tabular}
  \caption{$L^2$ and $H^1$ errors for interface concentration $v_h$ and iteration numbers on different refinement levels.\label{tab:ifaceError}}
\end{center}
\end{table}

The number of GCR iterations to reach a residual norm below $10^{-10}$ are reported in Table~\ref{tab:ifaceError}. Although, due to the small areas in certain cut elements, the  condition numbers of the stiffness matrices can be extremely large, we do not observe a strong deterioration in the iteration numbers of the preconditioned Krylov solver. The iteration numbers in Table~\ref{tab:ifaceError} show a rather regular  $\sim h^{-1}$ growth behavior. 

\subsection{Effect of small desorption}
The theory presented indicates that both the model and the discretization are stable if the desorption coefficients tend to zero. For the simpler case with only one bulk domain such a uniform stability result has not been derived in \cite{BurmanHansbo,ElliottRanner}. The analysis of well-posedness in  \cite{ElliottRanner} for the case of one bulk domain makes essential use of the assumption that the desorption coefficient is strictly positive and the constants in the analysis may blow up if the desorption coefficient tends to zero. In the numerical experiments presented in \cite{BurmanHansbo,ElliottRanner} only $k_a=k_d=1$ is considered. In applications one typically has $0 <k_d \ll k_a$. Therefore
 we include results of an experiment with a small or even vanishing desorption coefficient. For the bulk concentration, homogeneous Dirichlet boundary data on $\partial\Omega$ are chosen. The source terms are set to $f_i=0$, $i=1,2$ and $g=1$, so bulk concentration can only be generated by desorption of interface concentration from $\Gamma$. The material parameters are  chosen as $\nu_1=0.5$, $\nu_2=1$, $\nu_\Gamma=1$ and $\tilde k_{1,a}=\tilde k_{2,a}=\tilde k_{2,d}=1$, $\tilde k_{1,d}=\varepsilon$ with $\varepsilon\geq0$. We use the same initial triangulation as before. This initial mesh is uniformly refined 3 times, and the discrete problem is solved on this mesh for different values of $\varepsilon$, yielding solutions $u_{i,h}^\varepsilon\in V_{i,h}$, $v_h^\varepsilon\in V_{\Gamma,h}$.
Table~\ref{tab:desorpEps} shows the mean bulk concentration of $u_{1,h}^\varepsilon$ in $\Omega_{1,h}$,
\begin{align*}
\bar u_{1,h}(\varepsilon):= |\Omega_{1,h}|^{-1}\int_{\Omega_{1,h}} u_{1,h}^\varepsilon\,dx,
\end{align*}
for different values of the desorption coefficient $\tilde k_{1,d}=\varepsilon$. We  clearly observe  a linear behavior, which can be expected, based on the relation
\begin{align*}
  \tilde k_{1,a}\int_\Gamma u_1 \,ds=\int_\Gamma u_1 \,ds &= \tilde k_{1,d}\int_\Gamma v \,ds,
\end{align*}
that holds  for the continuous solution   $u_1,v$ (follows from the first and third equation in \eqref{Total}). In the numerical experiments the discrete analogon of this relation turns out to be satisfied
with value  $\int_{\Gamma_h} v_h \,ds = 17.775$ for all considered values of $\varepsilon$.
For $\varepsilon=0$ we have $\bar u_{1,h}(0)=1.37\cdot 10^{-15}$ which is due to round-off errors and the chosen tolerance $\mathrm{tol} = 10^{-14}$ of the iterative solver.  Hence,  in the  discrete problem the (mean) bulk concentration  in $\Omega_1$ is very close to zero (which is the correct value from the continuous model), i.e, there is no ``numerical leakage'' of surfactant concentration through the interface. The iteration numbers are essentially independent of  $\varepsilon$ ($155\pm1$ iterations for the considered range of $\varepsilon$), indicating that the condition number of the resulting preconditioned system matrix is robust w.r.t.\ $\tilde k_{1,d}\to 0$. These results illustrate the well-posedness of the model and the stability of the discretization for $\tilde k_{1,d} \downarrow 0$.
\begin{table}
  \begin{minipage}[t]{0.45\textwidth}
    \begin{center}
    \begin{tabular}{cc}
      \hline
      $\tilde k_{d,1}$ & $\bar u_{1,h}(\tilde k_{d,1})$ \\
      \hline\\
      1E+0 & 1.42E+00 \\
      1E-1 & 1.42E-01 \\
      1E-3 & 1.42E-03 \\
      1E-5 & 1.42E-05 \\
      1E-10 & 1.42E-10 \\
      0 & 1.37E-15 \\
      \hline\\
    \end{tabular}
  \end{center}
  \caption{\noindent Mean bulk concentration in $\Omega_{1,h}$ for different values of the desorption coefficient $\tilde k_{d,1}$ for refinement level 3.\label{tab:desorpEps}}
  \end{minipage}\hfill
  \begin{minipage}[t]{0.45\textwidth}
  \begin{center}
    \begin{tabular}{cc}
      \hline
      \# ref. & $\bar u_{1,h}(10^{-3})$ \\
      \hline\\
      0 & 1.3191E-03 \\
      1 & 1.3865E-03 \\
      2 & 1.4088E-03 \\
      3 & 1.4153E-03 \\
      4 & 1.4171E-03 \\
      \hline\\
    \end{tabular}
    \end{center}
    \caption{\noindent Mean bulk concentration in $\Omega_{1,h}$ with desorption coefficient $\tilde k_{d,1}=10^{-3}$ for different refinement levels.\label{tab:desorpH}}
  \end{minipage}
\end{table}

We now fix the value $\varepsilon=10^{-3}$ and change the number of uniform grid refinements, obtaining different mesh sizes $h_i$, $i=0,1,\ldots,4$. The obtained values of $\bar u_{1,h_i}(10^{-3})$ are given in Table~\ref{tab:desorpH}. These numbers show a linear convergence behavior with a contraction factor $\sim 0.3$, which indicates that we have stable (almost second order) convergence w.r.t  $h$.
On refinement level 3 we have an estimated relative error of approximately $2\cdot 10^{-3}$.

\section{Discussion and outlook} \label{sectdisc}
In this paper  a coupled system of elliptic partial differential equations is studied, in which two advection-diffusion equations in bulk subdomains are coupled, via adsorption-desorption terms, with an advection-diffusion equation on the interface between these bulk domains. This system of equations is motivated by models for surfactant transport in two-phase flow problems. A main result  is the well-posedness of a certain weak formulation of this system of equations.  We introduce an unfitted finite element method for the discretization of this problem. The method uses three trace spaces of one standard bulk finite element space. The interface is approximated by the zero level of a finite element function. For this finite element function and for the finite element functions used in the bulk space piecewise linears as well as higher order polynomials can be used. For this discretization method optimal error bounds are derived. We consider the following topics to be of interest for
 future research. \\ The unfitted finite element discretization leads to linear systems of equations that may become ill-conditioned. The discretization method treated in this paper can be combined with a stabilization procedure as recently presented in \cite{BurmanHansbo}. For the case of one bulk domain the method treated in this paper is the same as the unstabilized method in  \cite{BurmanHansbo}. The stabilization procedure considered in \cite{BurmanHansbo} does not rely on the fact that there is only one bulk domain, and looking at the analysis of the stabilized method in \cite{BurmanHansbo} we expect that it will work for the case with two bulk domains, too. \\ The use of an unfitted finite element technique becomes particularly attractive for time dependent problems with an evolving interface. In the applications that we have in mind (surfactants in two-phase incompressible flow) such time dependent problems are highly relevant. Hence, the extension of the method studied in this paper to a time-
dependent
coupled system as in \eqref{total} is an interesting topic. Such an extension may be based on space-time trace finite element methods that are studied in the recent papers \cite{ORXsinum,ORsinum2}, which deal with advection-diffusion equations on evolving surfaces, but without a coupling to a bulk phase. \\  Finally we note that if for the interface approximation a finite element polynomial of degree $q \geq 2$ is used, the zero level of this function is not directly available. A suitable quadrature for approximating integrals of $\Gamma_h$ has to be developed. This is not straightforward. For special cases in which the distance function to $\Gamma$ is known, computable parametrizations as treated in \cite{Demlow09} can be used. A more general approach has recently been presented in \cite{refJoergAR}. A similar quadrature problem arises for the tetrahedra in the bulk domain that are cut by the interface $\Gamma_h$.
\bibliographystyle{siam}
\bibliography{literatur}
 \end{document}